\documentclass{amsart}[10pt,a4paper,twoside]

\usepackage[T1]{fontenc}
\usepackage[utf8]{inputenc}
\usepackage{amsmath,amssymb,fancyhdr,enumerate,bbm}
\usepackage{hyperref}
\usepackage[all]{xy}
\usepackage[active]{srcltx}
\usepackage[english]{babel}

\theoremstyle{plain}
\newtheorem*{prop}{Proposition}
\newtheorem{Thm}{Theorem}

\newtheorem*{thm}{Theorem}

\newtheorem*{lem}{Lemma}
\theoremstyle{definition}
\newtheorem*{defn}{Definition}
\theoremstyle{remark}

\newcommand{\ind}{\mathrm{ind}\,}
\newcommand{\op}{\mathrm{op}}
\newcommand{\soc}{\mathrm{soc}}
\renewcommand{\hom}{\mathrm{Hom}}
\renewcommand{\mod}[1]{\mathrm{mod}\,#1}
\newcommand{\rad}{\mathrm{rad}}

\newcommand{\End}{\mathrm{End}}
\newcommand{\irr}{\mathrm{irr}}
\newcommand{\AR}[1]{\Gamma(\mod A)}
\newcommand{\id}{\mathrm{Id}}
\def\k{\mathbbm k}

\title[Covering techniques in Auslander-Reiten theory]%
{Covering techniques in Auslander-Reiten theory}
\author[C. Chaio]{Claudia Chaio}
\address[Claudia Chaio]{Centro Marplatense de Investigaciones Matemáticas, FCEyN, 
Universidad Nacional de Mar del Plata, CONICET. Funes 3350, 7600 Mar
del Plata, Argentina}
\email{claudia.chaio@gmail.com}
\thanks{The first and third named author thankfully acknowledge
financial
  support from CONICET and Universidad Nacional de Mar del Plata,
  Argentina. The first and third named authors are CONICET researchers}

\author[P. Le Meur]{Patrick Le Meur}
\address[Patrick Le Meur]{
Laboratoire de Math\'ematiques, Universit\'e Blaise Pascal \&
  CNRS, Complexe Scientifique Les Cézeaux, BP 80026, 63171 Aubi\`ere
  cedex, France}

\curraddr{Universit\'e Paris Diderot, Sorbonne Universit\'e, CNRS, Institut de
   Math\'ematiques de Jussieu-Paris Rive Gauche, IMJ-PRG, F-75013, Paris, France}

 \email{patrick.le-meur@imj-prg.fr}

\author[S. Trepode]{Sonia Trepode}
\address[Sonia Trepode]{Centro Marplatense de Investigaciones Matemáticas, FCEyN, 
Universidad Nacional de Mar del Plata, CONICET. Funes 3350, 7600 Mar
del Plata, Argentina}
\email{strepode@gmail.com}

\date{\today}

\begin{document}

\begin{abstract}
  Given a finite dimensional algebra over a perfect field the text
  introduces covering functors over the mesh category of any modulated
  Auslander-Reiten component of the algebra. This is applied to study
  the composition of irreducible morphisms between indecomposable
  modules in relation with the powers of the radical of the module
  category.
\end{abstract}

\maketitle

\section*{Introduction}

Let $A$ be a finite-dimensional algebra over a field $\k$. The
representation theory of $A$ deals with the category $\mathrm{mod}\,A$ of
finitely generated (right) $A$-modules. In particular it aims at
describing the indecomposable modules up to isomorphism and the
morphisms between them. For this purpose the Auslander-Reiten theory
gives useful tools such as irreducible morphisms and almost split
sequences. These two particular concepts have been applied to study
singularities of algebraic varieties and Cohen-Macaulay modules over
commutative rings.

Let $\mathrm{ind}\,A$ be the full subcategory of $\mathrm{mod}\,A$
containing one representative of each isomorphism class of
indecomposable $A$-modules. Given $X,Y\in \mathrm{ind}\,A$, a morphism
$f\colon X\to Y$ is called \emph{irreducible} if it lies in
$\rad\backslash\rad^2$. Here $\rad$ denotes the radical
of the module category, that is, the ideal in $\mathrm{mod}\,A$ generated
by the non-isomorphisms between indecomposable modules. The powers
$\rad^{\ell}$ of the radical are recursively defined by
$\rad^{\ell+1}=\rad^{\ell}\cdot \rad=\rad\cdot
\rad^{\ell}$.
The Auslander-Reiten theory encodes part of the information of
$\mathrm{mod}\,A$ in the Auslander-Reiten quiver $\Gamma(\mathrm{mod}\,A)$.
This concentrates much of the combinatorial information on the
irreducible morphisms and almost split sequences.  However it does not
give a complete information on the composition of two (or more)
irreducible morphisms. For example the composition of $n$ irreducible
morphisms obviously lies in $\rad^n$ but it may lie in
$\rad^{n+1}$. It is proved in \cite[Thm. 13.3]{MR748231} that if
these irreducible morphisms form a sectional path then their
composition lies in $\rad^n\backslash\rad^{n+1}$.  This
result was made more precise for finite-dimensional algebras over
algebraically closed fields in a study \cite{MR2819689} of the degrees
of irreducible morphisms (in the sense of \cite{MR1157550}) and their
relationship to the representation type of the algebra. The results in
\cite{MR2819689} are based on well-behaved functors introduced first
in \cite{MR576602,MR643558} for (selfinjective) algebras of finite
representation type.  This text presents general constructions of
well-behaved functors with application to composition of irreducible
morphisms.

Let $\Gamma$ be a connected component of the Auslander-Reiten quiver
of $A$ (or, an Auslander-Reiten component, for short). Let
$\mathrm{ind}\,\Gamma$ be the full subcategory of $\mathrm{ind}\,A$
with set of objects the modules $X\in\mathrm{ind}\,A$ lying in
$\Gamma$. Beyond the combinatorial structure on $\Gamma$, the
mesh-category $\k(\Gamma)$ is a first approximation of
$\mathrm{ind}\,\Gamma$ taking into account the composition of
irreducible morphisms. Actually Igusa and Todorov have shown that
$\Gamma$ comes equipped with a $\k$-modulation (\cite{MR748232}),
which is called \emph{standard} here and which includes the division
algebra $\kappa_X=\mathrm{End}_A(X)/\rad(X,X)$ and the
$\kappa_X-\kappa_Y$-bimodule $\mathrm{irr}(X,Y)=\rad(X,Y)/\rad^2(X,Y)$
for every $X,Y\in\Gamma$. The category $\k(\Gamma)$ may be defined by
generators and relations (see Section~\ref{sec:prel} for details). Its
objects are the modules $X\in\Gamma$, the generators are the classes
of morphisms $u\in \kappa_X$ (as morphisms in $\k(\Gamma)(X,X)$) and
$u\in\mathrm{irr}(X,Y)$ (as morphisms in $\k(\Gamma)(X,Y)$), for every
$X,Y\in\Gamma$, and the ideal of relations is the mesh ideal.

When $\k$ is a perfect field this text introduces a covering functor
of $\mathrm{ind}\,\Gamma$ in order to get information about the composition of
irreducible morphisms in $\Gamma$.

The Auslander-Reiten component is called \emph{standard} if there exists an
isomorphism of categories $\k(\Gamma)\simeq \mathrm{ind}\,\Gamma$. Not
all Auslander-Reiten components are standard and in many cases there
even exist no functor $\k(\Gamma)\to \mathrm{ind}\,\Gamma$. For
instance if $\Gamma$ has oriented cycles then such a functor is likely
not to exist. This may be bypassed replacing the mesh category
$\k(\Gamma)$ by that of a suitable translation quiver
$\widetilde\Gamma$ with a $\k$-modulation such that there exists a
covering $\pi\colon \widetilde\Gamma\to \Gamma$. It appears that the
composition of irreducible morphisms in $\mathrm{ind}\,\Gamma$ may be
studied using $\k(\widetilde{\Gamma})$ provided that there exists a
so-called well-behaved functor
$\k(\widetilde\Gamma)\to \mathrm{ind}\,\Gamma$.  These functors were
first considered by Riedtmann \cite{MR576602} in her study of the
shapes of the Auslander-Reiten quivers of selfinjective algebras of
finite representation type over algebraically closed fields and next
by Bongartz and Gabriel \cite{MR643558} for algebras of finite
representation type over algebraically closed fields. The present
article considers well-behaved functors for modulated translation
quivers when $\k$ is a perfect field. Let $(\kappa_x,M(x,y))_{x,y}$ be
the $\k$-modulation of $\widetilde\Gamma$ induced by the standard
modulation of $\Gamma$, that is, $\kappa_x=\kappa_{\pi x}$ and
$M(x,y)=\mathrm{irr}(\pi x,\pi y)$ for every
$x,y\in \widetilde\Gamma$. Then a functor
$F\colon \k(\widetilde\Gamma)\to \mathrm{ind}\,\Gamma$ is well-behaved
if it induces isomorphisms $\kappa_x\simeq\kappa_{\pi x}$ and
$M(x,y)\simeq \mathrm{irr}(\pi x,\pi y)$, for every $x,y\in\Gamma$.
The construction of $F$ relies on three fundamental facts.  Firstly,
if one tries to construct such an $F$ then it is quite natural to
proceed by induction. The translation quiver $\widetilde\Gamma$ is
called \emph{with length} if any two paths in $\Gamma$ having the same
source and the same target have the same length. As mentionned above
an inductive construction is likely not to work if $\widetilde\Gamma$
has oriented cycles and actually simple examples show that this
construction fails if $\widetilde\Gamma$ is not with length. Note that
$\widetilde\Gamma$ is with length when $\widetilde\Gamma$ is the
universal cover of \cite{MR643558}. Secondly, if
$x\in\widetilde\Gamma$ then the ring homomorphism
$\kappa_x\hookrightarrow \k(\widetilde\Gamma)(x,x)\xrightarrow F
\mathrm{End}_A(\pi x)$
is a section of the quotient homomorphism
$\mathrm{End}_A(X)\twoheadrightarrow \kappa_x$. In view of the
Wedderburn-Malcev theorem this section is most likely to exist in the
framework of algebras over perfect fields. Finally, given an
irreducible morphism $f\colon X\to Y$ with $X,Y\in\Gamma$ then there
exist $x,y\in\widetilde\Gamma$ and $u\in\k(\widetilde\Gamma)(x,y)$
such that $f-Fu\in\rad^2$. In view of studying the composition of
irreducible morphisms in $\mathrm{ind}\,\Gamma$ one may wish to have
an equality $f=Fu$. This would permit to \emph{lift} the study into
$\k(\widetilde\Gamma)$ where the composition of morphisms is better
understood because of the mesh ideal. Keeping in mind these comments
the main result of this text is the following.

\begin{Thm}
  \label{thm1}
  Let $A$ be a finite-dimensional
  algebra over a perfect field $\k$. Let $\Gamma$ be an Auslander-Reiten component of $A$. Let
  $\pi\colon \widetilde\Gamma\to \Gamma$ be a covering of translation
  quivers where $\widetilde\Gamma$ is with length.
 There exists a well-behaved functor $F\colon \k(\widetilde
    \Gamma)\to \mathrm{ind}\,\Gamma$.
\end{Thm}

The study of the composition of irreducible mophisms in
$\mathrm{ind}\,\Gamma$ using such a covering functor $F$ is made possible
by the following lifting (or, covering) property of $F$ which is the
second main result of the text. No assumption is made on length.
\begin{Thm}
  \label{thm2}
% MODIFICATION : FOR CLARITY THE SETTING IS RECALLED
  Let $A$ be a finite-dimensional algebra over a perfect field
  $\k$. Let $\Gamma$ be an Auslander-Reiten component of $A$. Let
$\pi\colon \widetilde\Gamma\to \Gamma$ be a covering of translation
quivers.
  Let $F\colon \k(\widetilde{\Gamma})\to\mathrm{ind}\,\Gamma$ be a
  well-behaved functor, $x,y\in\widetilde\Gamma$
 and let $n\geqslant 0$. 
\begin{enumerate}[(a)]
\item The two following maps induced by $F$ are bijective
\[
   \begin{array}{rclc}
     \bigoplus\limits_{Fz=Fy}\mathfrak{R}^n\k(\widetilde{\Gamma})(x,z)
     /
\mathfrak{R}^{n+1}\k(\widetilde{\Gamma})(x,z)
     & \to &
     \rad^n(Fx,Fy)/\rad^{n+1}(Fx,Fy) \\
\\
     \bigoplus\limits_{Fz=Fy}\mathfrak{R}^n\k(\widetilde{\Gamma})(z,x)/
     \mathfrak{R}^{n+1}\k(\widetilde{\Gamma})(z,x)
     & \to &
     \rad^n(Fy,Fx)/\rad^{n+1}(Fy,Fx) & .
   \end{array}
\]
\item The two following maps induced by $F$ are injective
\[
     \bigoplus\limits_{Fz=Fy}\k(\widetilde{\Gamma})(x,z) \to
     \mathrm{Hom}_A(Fx,Fy)
\ \ and\ \ 
     \bigoplus\limits_{Fz=Fy}\k(\widetilde{\Gamma})(z,x) \to 
    \mathrm{Hom}_A(Fy,Fx) .
\]
\item $\Gamma$ is generalized standard if and only if $F$ is a covering
  functor, that is, the two maps of (b) are bijective (see \cite[3.1]{MR643558}).
\end{enumerate}  
\end{Thm}
% MODIFICATION: THE RADICAL HAS TO BE DEFINED IN THE INTRODUCTION.
Here $\mathfrak R\k(\widetilde \Gamma)$ is the ideal in
$\k(\widetilde\Gamma)$ generated by the morphisms in $M(x,y)$, for
every arrow $x\to y$ in $\widetilde\Gamma$. Call
it the \emph{radical} of
$\k(\widetilde\Gamma)$ by abuse of terminology. Define  its powers $\mathfrak
R^n\k(\widetilde\Gamma)$ like for the radical of $\mathrm{mod}\,A$.
Here is an interpretation of Theorem~\ref{thm2}.
Both $\k(\widetilde\Gamma)$ and $\mathrm{ind}\,\Gamma$ are filtered by the powers of their respective
radicals. The above theorem asserts that $F$ induces a covering functor $\mathrm{gr}\,\k(\widetilde\Gamma)\to \mathrm{gr}\,\mathrm{ind}\,\Gamma$ (in the
sense of \cite{MR643558}) between
the associated graded categories.

This text is therefore organised as follows. Section~\ref{sec:prel} is
a reminder on Auslander-Reiten theory, modulated translation quivers
and their mesh-categories, and coverings of translation quivers; it
also introduces strongly irreducible morphisms. Section~\ref{sec:wb}
proves the above theorems. Section~\ref{sec:applications1} gives an
application to the composition of irreducible morphisms.

In the sequel $\k$ denotes a perfect field.  Hence the tensor product
over $\k$ of two finite-dimensional division algebras is semi-simple.
Also if $R$ is a finite-dimensional $\k$-algebra and $J\subseteq R$ is
a two-sided ideal such that $R/J$ is a division $\k$-algebra, then the
natural surjection $R\twoheadrightarrow R/J$ admits a section
$R/J\hookrightarrow R$ as a $\k$-algebra. The composition of
mappings $f\colon X \to Y$ and $g \colon Y \to Z$ is denoted by $fg
\colon X \to Z$.

\section{Preliminaries}
\label{sec:prel}

\subsection{Reminder on Auslander-Reiten theory}
\label{sec:remind-ausl-reit}

Following \cite[V.7, p.178]{MR1476671}, denote by $\rad$ the
\emph{radical} of $\mod A$. This is the ideal of $\mod A$ generated by the
non invertible morphisms between indecomposable modules. Its powers
are defined recursively by $\rad^0=\mod A$ and
$\rad^{n+1}=\rad\cdot\rad^n$ ($=\rad^n\cdot \rad$) for every integer
$n\geqslant 0$.  For $X\in \mod A$ indecomposable, the division algebra
$\End_A(X)/\rad(X,X)$  is denoted by
$\kappa_X$ and, given $u\in \End_A(X)$, its residue class in $\kappa_X$ is
denoted by $\overline u$.

A morphism $f$ in $\mod A$ is  \emph{irreducible} (\cite[V.5,
p. 166]{MR1476671}) if it is neither a section nor a retraction and
for all decompositions $f=uv$, then $u$ is a section or else $v$ is a
retraction. Given $X,Y\in \mod A$ indecomposable, a morphism $X\to Y$
is irreducible if and only if it belongs to $\rad$ and not to $\rad^2$,
and the quotient $\rad(X,Y)/\rad^2(X,Y)$ is called the \emph{space of
  irreducible morphisms} from $X$ to $Y$ and denoted by
$\irr(X,Y)$. Given $f\in \rad(X,Y)$, its image in
$\irr(X,Y)$ is denoted by $\overline f$. Hence,
$\irr(X,Y)$ is a $\kappa_X-\kappa_Y$-bimodule, or, a left
$\kappa_X\otimes_\k\kappa_Y^{\op}$-module, such that
$\overline u \cdot \overline f \cdot \overline v = \overline{ufv}$ for
all $u\in \End_A(X)$, $f\in \rad(X,Y)$ and $v\in \End_A(Y)$.

A \emph{minimal left almost split} morphism (\cite[V.1,
137-138]{MR1476671}) is a morphism $f\colon X\to Y$ in $\mod A$ that
is not a section, such that any morphism with domain $X$ and which is
not a section factors through $f$ and such that, for all
$u\in \End_A(X)$, if $uf=f$ then $u$ is invertible. Such a morphism
features the following properties (\cite[V.1, VII.1 Prop. 1.3
p. 230]{MR1476671}): $X$ is indecomposable and, given any direct sum
decomposition $Y = \oplus_{i=1}^rX_i^{n_i}$ where $X_1,\ldots,X_r$ are
indecomposable and pairwise non isomorphic, if $f$ is denoted
coordinatewise as $[f_{i,j}\ ;\
\begin{subarray}{l}
  1\leqslant i\leqslant r \\ 1\leqslant j \leqslant n_i
\end{subarray}]$,
then, for all $i\in \{1,\ldots,r\}$, the family
$\{\overline{f_{i,j}}\}_j$ is a basis of $\irr(X,X_i)$ over
$\kappa_{X_i}$. These properties characterise minimal left almost
split morphisms. Any indecomposable $A$-module $X$ is the domain of such a
morphism $X\to Y$ which happens to be the quotient morphism $X\to X/\soc(X)$
when $X$ is injective and, in general, any irreducible morphism $X\to
Z$ is equal to the composite morphism $X\to Y \to Z$ for some
retraction $Y\to Z$.  Minimal right almost split morphisms are defined
dually and feature dual properties.

An \emph{almost split sequence} (\cite[V.1, p.144]{MR1476671}) is an
exact sequence $0\to X\to Y\to Z\to 0$ in $\mod A$ such that $X\to Y$
is minimal left almost split and $Y\to Z$ is minimal right almost
split. Any indecomposable and non-projective $Z\in \mod A$ (or,
non-injective $X\in \mod A$) is the end term (or, the first term,
respectively) of such a (unique up to isomorphism) sequence
(\cite[V.1, Thm. 1.15 p. 145]{MR1476671}), in such a case, $X$ is
called the \emph{Auslander-Reiten translate} of $Z$ and denoted by
$\tau_AZ$.

The \emph{Auslander-Reiten} quiver of $\mod A$ (\cite[VII.1
p. 225]{MR1476671}) is the pair $(\AR A,\tau_A)$ where $\AR A$ is the
quiver with vertices the objects of $\ind A$ and such that for all
$X,Y\in \ind A$ there is an arrow $X\to Y$ if and only if there exists
an irreducible morphism $X\to Y$. Its connected components are called
\emph{Auslander-Reiten components}.  These quivers are particular
instances of \emph{translation quivers} (\cite[VII.4]{MR1476671}
p. 248). Those are the pairs $(\Gamma,\tau)$ where $\Gamma$ is a
locally finite quiver with two distinguished sets of vertices called
\emph{projectives} and \emph{injectives}, respectively, and $\tau$,
the \emph{translation}, is a bijective mapping from non-projective
vertices to non-injective vertices and such that such that for all
vectices $x,y$ with $x$ non-projective, there is an arrow $y\to x$ if
and only if there is an arrow $\tau x\to y$.

The \emph{valuation} $(a,b)$ of an arrow $X\to Y$ in $\AR A$ is
defined such that $a$ is the maximal integer such that
there exists a minimal right almost split morphism $X^a \oplus Z\to Y$
and $b$ is the maximal integer such that there exists a minimal left
almost split morphism $X\to Y^b \oplus Z$. Equivalently,
$(a,b) = (\dim_{\kappa_X}\irr(X,Y),\dim_{\kappa_Y^\op}\irr(X,Y))$.

\subsection{Strongly irreducible morphisms}
\label{subsec_modules}

The following definition is given over arbitrary fields. However, it
is relevant mostly when $\k$ is a perfect field, see below.
\begin{defn}
  Let $\k$ be any field. Let $A$ be a finite dimensional
  $\k$-algebra. Let $X,X_1,\ldots,X_r\in \ind A$ be such that
  $X_1,\ldots,X_r\in \mathrm{ind}\,A$ are pairwise non isomorphic and
  let $n_1,\ldots,n_r$ be positive integers. A \emph{strongly
    irreducible} morphism $X\to \oplus_{i=1}^rX_i^{n_i}$ is an
  irreducible morphism
  $\left[f_{i,j};\,1\leqslant i\leqslant r,\,1\leqslant j\leqslant
    n_i\right]$
  with the following properties in the left
  $\kappa_X\otimes_\k\kappa_{X_i}^\op$-module $\irr(X,X_i)$, for all
  $i\in \{1,\ldots,r\}$.
  \begin{enumerate}[(a)]
  \item
    $\sum_j\kappa_X \cdot \overline{f_{i,j}}\cdot \kappa_{X_i} =
    \oplus_j\kappa_X \cdot
    \overline{f_{i,j}} \cdot \kappa_{X_i}$.
  \item
    $\sum_j\kappa_X \cdot
    \overline{f_{i,j}} \cdot \kappa_{X_i}$
    is a direct summand of $\irr(X,X_i)$ as a left
    $\kappa_X\otimes_\k\kappa_{X_i}^\op$-module.
  \end{enumerate}
\end{defn}
Note that, $f$ is strongly irreducible if and only if, for all $i$, so
is $X\to X_i^{n_i}$. As mentioned previously, this definition is
mostly relevant when $\k$ is a perfect field because of the following
standard result.
\begin{lem}
  Let $\k$ be a perfect field.
  \begin{enumerate}
  \item Let $E,F$ be finite dimensional division $\k$-algebras. Then,
    $E\otimes_\k F$ is semi-simple.
  \item Let $A$ be a finite dimensional $\k$-algebra. Let $X,Y\in \ind
    A$. Then, $\kappa_X\otimes_\k \kappa_Y^\op$ is semi-simple.
  \end{enumerate}
\end{lem}
\begin{proof}
  ($1$) follows from \cite[10.7 Corollary of
  p. 188]{MR674652} and \cite[10.7 Corollary b
  p. 192]{MR674652}). ($2$) follows from ($1$).
\end{proof}

Here are some comparison elements between strongly irreducible
morphisms and irreducible morphisms, $\k$ here need not be
perfect. Keep the notation used in the previous definition. First,
condition (b) holds automatically when $\k$ is perfect. Next, $f$ is
irreducible if and only if so is each $f_{i,j}$ and the following
assertions are true in the right $\kappa_{X_i}$-vector space
$\irr(X,X_i)$, for all $i\in \{1,\ldots,r\}$.
  \begin{enumerate}
  \item[(a')]
    $\sum_j\overline{f_{i,j}} \cdot \kappa_{X_i} = \oplus_j
    \overline{f_{i,j}} \cdot \kappa_{X_i}$.
  \item[(b')] $\sum_j \overline{f_{i,j}} \cdot \kappa_{X_i}$ is a direct summand of $\irr(X,X_i)$.
  \end{enumerate}
  Note that condition (b') is independent of $f$ because $\kappa_{X_i}$
  is a division algebra.  Clearly, conditions (a) and (b) in the above
  definition are stronger than conditions (a') and (b'). In
  particular, there are irreducible morphisms which are not strongly
  irreducible. Here is a example (see \cite[VII.2,
  p. $235$]{MR1476671}).  Let $\k=\mathbb R$ and
  \[
  A=\left\{\left(
      \begin{smallmatrix}
        a & 0 \\ b & c
      \end{smallmatrix}\right)\ |\ a\in \mathbb C,\,b\in \mathbb
    C,\,c\in \mathbb R\right\}\,.
  \]
  Then, $\AR A$ has the following shape, as a valued quiver,
  \[
  \def\objecstyle{\scriptstyle}
  \def\labelstyle{\scriptscriptstyle}
  \xymatrix@=5pt{
    && |P \ar[rrd]^{(1,2)} \ar@{.}[rrrr] && && S'|
    \\
    |S \ar[rru]^{(2,1)} \ar@{.}[rrrr] && && I| \ar[rru]_{(2,1)}
  }
  \]
  where $S$, $P$, $I$, $S'$ are simple projective, projective non
  simple, injective non simple and simple injective,
  respectively. Consider an almost split sequence
  \[
  0 \to S \xrightarrow{f} P \xrightarrow{g}
  I\to 0\,.
  \]
  Then, there exists an automorphism $u\colon P \to
  P$ fitting in an almost split sequence
  \[
  0 \to P \xrightarrow{[g ,\ ug]} I^2 \to S'\to 0\,.
  \]
  The valuation of the arrow $P\to I$ of $\AR A$ is $(1,2)$. Therefore
  $\dim_{\kappa_P}\irr(P,I)=1$.  Consequently, $\irr(P,I)$ is
  generated by a single element as a left $\kappa_P$-vector space, and
  hence also as a left
  $\kappa_P\otimes_{\mathbb R}\kappa_I^\op$-module. Accordingly, the
  left $\kappa_P\otimes_{\mathbb R}\kappa_I^\op$-module $\irr(P,I)$ is
  simple. Thus, it equals both
  $\kappa_P\cdot \overline g\cdot \kappa_I$ and
  $\kappa_P\cdot\overline{ug}\cdot\kappa_I$.  In particular,
  $g\colon P\to I$ is strongly irreducible; and $[g\ ug]$ is not
  strongly irreducible despite being irreducible, because the sum
  $\kappa_P\cdot \overline g \cdot \kappa_I + \kappa_P \cdot
  \overline{ug}\cdot \kappa_I$ is not direct.

  From now on, assume that $\k$ is perfect. Here are hints on how to
  construct strongly irreducible morphisms and how to obtain
  irreducible morphisms from them.  Consider an arrow $X\to Y$ in
  $\AR A$. Since $\kappa_X\otimes_\k\kappa_Y^\op$ is a semi-simple
  algebra, there is a direct sum decomposition
  $\irr(X,Y) = \oplus_{j=1}^n S_j$, where $S_1,\ldots,S_n$ are simple
  left $\kappa_X\otimes_\k\kappa_Y^\op$-modules. For each $j$, let
  $f_j\in \rad(X,Y)$ be such that $\overline{f_j}$ is a generator of
  the simple module $S_j$. Then, the following morphism is strongly
  irreducible
\begin{equation}
  \label{eq:2}
\left[f_j;\,1\leqslant j \leqslant n\right]\colon X\to Y^n\,.
\end{equation}
Any strongly irreducible
morphism $X\to Y^m$ ($m\geqslant 1$) is equal to the composition of a
morphism such as (\ref{eq:2}), for adequate choices, with a retraction
$Y^n\to Y^m$. Now, for every $j\in\{1,\ldots,n\}$, consider $S_j$ as a
right $\kappa_Y$-vector space; then, using a basis of $\kappa_X$ over
$\k$ it follows that there exists a family $\{u_{j,k}\}_k$ of
automorphisms of $X$ with at most $\dim_\k\kappa_X$ terms and such
that
\[
S_j = \oplus_k \overline{u_{j,k}f_j}\cdot \kappa_Y\,.
\]
Accordingly, there is a direct sum decomposition of right
$\kappa_Y$-vector spaces
\[
\irr(X,Y) =\oplus_{j,k}\overline{u_{j,k}f_j}\cdot \kappa_Y\,.
\]
Hence, this  sum has $\dim_{\kappa_Y}\irr(X,Y)$ terms and the
following morphism is irreducible,
\begin{equation}
  \label{eq:3}
  [u_{j,k}f_j;\,j,k] \colon X \to Y^{\dim_{\kappa_Y}\irr(X,Y)}\,.
\end{equation}
In particular, if the valuation of the arrow $X \to Y$ of $\AR A$ has
the shape $(a,1)$, then all the direct sums in (\ref{eq:2}) and
(\ref{eq:3}) have exactly one term. Actually, these considerations
show the equivalence of the following assertions.
\begin{enumerate}[(i)]
\item The irreducible morphism (\ref{eq:3}) is strongly irreducible.
\item $\irr(X,Y)$ is the direct sum of $\dim_{\kappa_Y} \irr(X,Y)$ simple
  $\kappa_X \otimes_\k \kappa_Y^\op$-modules.
\item For all $j\in \{1,\ldots,n\}$, the $\k$-vector spaces
  $\kappa_X \cdot \overline{f_j} \cdot \kappa_Y$ and
  $\overline{f_j} \cdot \kappa_Y$ are equal.
\item $\kappa_X\cdot \overline{f_j}$ is contained in
  $\overline{f_j}\cdot \kappa_Y$, for all $j\in \{1,\ldots,n\}$.
\end{enumerate}

Going back to the general situation of an irreducible morphism
$f\colon X \to \oplus_{i=1}^rX_i^{n_i}$, it would be worth
characterising when it is strongly irreducible. In view of the previous
considerations, each of the following conditions is sufficient.
\begin{enumerate}
\item $n_1=\cdots = n_r = 1$.
\item $\kappa_X \cdot \overline{f_{i,j}}$ is contained in
  $\overline{f_{i,j}}\cdot \kappa_{X_i}$, for all $i,j$; indeed, in
  such case, the sum
  $\sum_{i,j}\kappa_X \cdot \overline{f_{i,j}}\cdot \kappa_{X_i}$ is
  direct because the morphism $X\to \oplus_iX_i^{n_i}$ is irreducible
  and, for every $i,j$, the $\k$-vector spaces
  $\kappa_X \cdot \overline{f_{i,j}} \cdot \kappa_{X_i}$ and
  $\overline{f_{i,j}} \cdot \kappa_{X_i}$ are equal.
 \item $\kappa_X = \k$; indeed, this assertion implies (2).
\item $\k$ is algebraically closed; indeed, this assertion implies (3).
\item For each $i\in \{1,\ldots,r\}$, the arrow $X\to X_i$ of $\AR A$
  has valuation $(a_i,1)$; indeed, this entails (1).
\end{enumerate}
Recall  (\cite[1.7]{MR1157550}) that if an arrow $X\to Y$ of $\AR A$ has
finite left degree, then it has valuation $(1,b)$ or $(a,1)$. In the
former case, if  $b\geqslant 2$, then
there exists an irreducible morphism $X\to Y^b$ which is necessarily
not strongly irreducible.

\subsection{Factorisation through minimal almost split morphisms}
\label{subsec_factorisation}

The reader is referred to \cite{MR1476671} for basics on Auslander-Reiten
theory. In the sequel the factorisation property of minimal almost
split morphisms is used
as follows.
\begin{lem}
  Let $u\colon X\to Y$ be a left minimal almost split morphism,
  $Z\in \mathrm{mod}\,A$ and $v\in \rad^{n+1}(X,Z)$ for some
  $n\geqslant 0$. There exists $w\in \rad^n(Y,Z)$ such that $v=uw$.
\end{lem}
\begin{proof}
  Since $v\in \rad^{n+1}$, there exist an integer $N\geqslant 1$,
  indecomposable modules $X_1,\ldots,X_N$ and morphisms $v_i'\in
  \rad(X,X_i)$ and $v_i''\in \rad^n(X_i,Z)$, for every $i\in
  \{1,\ldots,N\}$, such that $v = \sum_{i=1}^N v_i'v_i''$. For every
  $i\in \{1,\ldots,N\}$, there exists $w_i\in \hom_A(Y,X_i)$ such that
  $v_i' = u w_i$. Let $w = \sum_{i=1}^N w_iv_i''$. Then, $w\in
  \rad^n(Y,Z)$ and $v=uw$.
\end{proof}

\subsection{Modulated translation quivers and their mesh-categories}
\label{subsec_modulation}

Let $\Gamma$ be a translation quiver. For a non-projective vertex $x$,
the subquiver of $\Gamma$ formed by the arrows starting in $\tau x$
and the arrows arriving in $x$ is called the \emph{mesh} starting in
$\tau x$.

A \emph{($\k$)-modulation} on $\Gamma$ is the following data
\begin{enumerate}[(i)]
\item a division $\k$-algebra $\kappa_x$ for every vertex $x\in
  \Gamma$,
\item a non-zero $\kappa_x-\kappa_y$ bimodule $M(x,y)$ for every arrow
  $x\to y$ in $\Gamma$,
\item a $\k$-algebra isomorphism $\tau_*\colon \kappa_x\xrightarrow{\sim}
  \kappa_{\tau x}$ for every vertex $x\in\Gamma$,
\item a non-degenerate $\kappa_y-\kappa_y$-linear map $\sigma_*\colon
  M(y,x)\otimes_{\kappa_x} M(\tau x,y)\to \kappa_y$ 
  (the left $\kappa_x$-module structure on $M(\tau x,y)$ is defined
  using its structure of left $\kappa_{\tau x}$-module and  $\tau_*\colon \kappa_x\to \kappa_{\tau x}$).
\end{enumerate}
With such a structure, $\Gamma$ is called a \emph{modulated}
translation quiver.  If $A$ is a finite-dimensional algebra over a
field $\k$ then the Auslander-Reiten quiver $\Gamma(\mathrm{mod}\,A)$
has a $\k$-modulation as follows (\cite[2.4, 2.5]{MR748232}), in the
rest of the text this modulation is called the \emph{standard}
modulation of $\AR A$; the restriction of the standard modulation of
$\AR A$ to any Auslander-Reiten component is also called the standard
modulation of that component.  For every non-projective
$X\in\mathrm{ind}\,A$ fix an almost split sequence
$0\to \tau_A X\to E\to X\to 0$ in $\mathrm{mod}\,A$.  Then
\begin{itemize}
\item $\kappa_X=\mathrm{End}_A(X)/\rad(X,X)$ for every $X\in \mathrm{ind}\,A$,
\item $M(X,Y)=\mathrm{irr}(X,Y)$
  for every arrow $X\to Y$ in $\Gamma(\mathrm{mod}\,A)$,
\item for every $X\in \mathrm{ind}\,A$ and every morphism $u\colon X\to
  X$ defining the residue class $\overline u\in \kappa_X$, let
  $\tau_*\overline u
  \colon \tau_A X\to \tau_AX$ be the residue class $\overline v$ where
  $v\colon \tau_A X\to \tau_A X$ is a morphism fitting into a
  commutative diagram
\[
    \xymatrix@R=15pt{
0 \ar[r] & \tau_A X \ar[r] \ar[d]_v & E \ar[r] \ar[d]& X \ar[r] \ar[d]^u& 0\\
0 \ar[r] & \tau_A X \ar[r]& E \ar[r]& X \ar[r]& 0\,,
}
\]
\item let $X,Y\in \mathrm{ind}\,A$ with $X$ non-projective and assume
  that there is an arrow $Y\to X$ in $\Gamma(\mathrm{mod}\,A)$. Let $\overline
  u\in M(\tau_A X,Y)$ and $\overline v\in M(Y,X)$ be residue classes
  of morphisms $u\colon \tau_A X\to Y$ and $v\colon Y\to X$,
  respectively. Then define $\sigma_*(\overline v\otimes \overline u)$
  as the composition $\overline{v'u'}$ where $u',v'$ are morphisms
  fitting into a commutative diagram
\[
    \xymatrix@R=15pt{
&&& Y \ar[ld]_{v'} \ar[d]^v\\
0 \ar[r] & \tau_A X \ar[d]_u \ar[r]  & E \ar[r] \ar[ld]^{u'} & X
\ar[r] & 0\\
&Y&&&.
}
\]
\end{itemize}
This construction does not depend on the initial choice of the almost
split sequences up to an isomorphism of modulated translation quivers
(\cite[2.5]{MR748232}). In the sequel $\Gamma(\mathrm{mod}\,A)$ is
considered as a modulated translation quiver as above.

If $\Gamma$ is a modulated translation quiver, its
\emph{mesh-category} $\k(\Gamma)$ is defined as follows
(\cite[1.7]{MR748232}). Let $S$ be the semi-simple category whose
object set is the set of vertices in $\Gamma$ and such that
$S(x,y)=\kappa_x$ if $x=y$ and $S(x,y)=0$ otherwise. The collection
$\{M(x,y)\}_{\text{$x\to y$ in $\Gamma$}}$ naturally defines an
$S-S$-bimodule denoted by $M$. The \emph{path-category} is the tensor
category $T_S(M)$ also denoted by $\k \Gamma$. The \emph{mesh-ideal}
is the ideal in $\k \Gamma$ generated by a collection of morphisms
$\gamma_x\colon \tau x\to x$ indexed by the non-projective vertices
$x\in\Gamma$. Given a non-projective vertex $x\in\Gamma$, a morphism
$\gamma_x\colon \tau x\to x$ in $\k \Gamma$ is defined as follows. For
every arrow $y\to x$ ending in $x$, fix a basis $(u_1,\ldots,u_d)$ of
the $\kappa_y$-vector space $M(y,x)$. Let $(u_1^*,\ldots,u_d^*)$ be
the associated dual basis of the $\kappa_y$-vector space $M(\tau x,y)$
under the pairing $\sigma_*$ (that is, $\sigma_*(u_i\otimes u_j^*)$ is
$1$ if $i=j$ and $0$ otherwise). Then
   % MODIFICATION: the definition of \gamma_x is made shorter to save
   % space
$\gamma_x=\sum\limits_{\text{\tiny $y\to x$ in
        $\Gamma$}}\sum\limits_iu_i^* u_i\in \k\Gamma(\tau x,x)$.
%   \begin{equation}
%     \gamma_x=\sum\limits_{\tiny \text{$y\to x$ in
%         $\Gamma$}}\sum\limits_iu_i^*\otimes u_i\in
% \bigoplus\limits_{y\to x\ \text{in}\ \Gamma}    M(\tau x,y)\underset{\kappa_y}{\otimes} M(y,x)\subseteq \k \Gamma(\tau x,x)\,.\notag
% \end{equation}
This morphism does not depend on the choice of the basis $(u_1,\ldots,u_d)$.
% MODIFICATION: the statement of the definition of mesh-category is
% made shorter.
The mesh-category  is then defined as the quotient
category of $\k\Gamma$ by the mesh-ideal.

Let $\Gamma$ be a component of $\Gamma(\mathrm{mod}\,A)$ endowed with
its standard modulation. As proved in \cite[Sect. 2]{MR748232}, the
mesh-category $\k(\Gamma)$ does not depend on the choice of the almost
split sequences used to define the modulation up to an isomorphism of
$\k$-linear categories. The following lemma explains how to recover
the mesh-relations $\gamma_X$ and the pairing $\sigma_*$ starting from
a different choice of almost split sequences.

\begin{lem}
  In the previous setting, let $X\in\Gamma$ be non-projective and let
  \[
    0\to \tau_A X\xrightarrow \alpha
    \bigoplus\limits_{i=1}^rX_i^{n_i}\xrightarrow \beta X\to 0
    \]
be the almost split sequence ending in $X$ that is used in the
definition of the modulation on $\Gamma$, where $X_1,\ldots,X_r\in
\Gamma$ are pairwise distinct. Let
\[
    0\to \tau_A X\xrightarrow{f=\left[
f_{i,j}\,;\,
\begin{smallmatrix}
  1\leqslant i\leqslant r\\
1\leqslant j\leqslant n_i
\end{smallmatrix}
\right]
}
    \bigoplus_iX_i^{n_i}\xrightarrow{
g=\left[
g_{i,j}\,;\,
\begin{smallmatrix}
  1\leqslant i\leqslant r\\
1\leqslant j\leqslant n_i
\end{smallmatrix}
\right]^t
} X\to 0
\]
be another almost split sequence where the $f_{i,j}\colon \tau_AX\to
X_i$ and the $g_{i,j}\colon X_i\to X$ are the components of
$f$ and $g$, respectively. Then there is a commutative diagram
\[
  \xymatrix@R=15pt{
0 \ar[r] & \tau_AX\ar[r]^\alpha \ar[d]_u &
\bigoplus_iX_i^{n_i}\ar[r]^\beta \ar[d]^v & X\ar[r] \ar@{=}[d] & 0\\
0 \ar[r] & \tau_AX\ar[r]^f &
\bigoplus_iX_i^{n_i}\ar[r]^g & X\ar[r]  & 0
}
\]
where $u,v$ are isomorphisms. With this data:
\begin{enumerate}[(a)]
\item $(\overline{g_{i,1}},\ldots,\overline{g_{i,n_i}})$ is a basis of
  the left $\kappa_{X_i}$-module $\mathrm{irr}(X_i,X)$ and the corresponding
  dual basis of  the right $\kappa_{X_i}$-module $\mathrm{irr}(\tau_AX,X_i)$ 
  under the pairing $\sigma_*$ is  $(\overline{uf_{i,1}},\ldots,\overline{uf_{i,n_i}})$, for
  every $i\in\{1,\ldots,r\}$,
\item $\gamma_X=\sum\limits_{i,j}\overline u\overline{f_{i,j}}\otimes
  \overline{g_{i,j}}$.
\end{enumerate}
\end{lem}
\begin{proof}
$(b)$ follows directly from $(a)$ and from the definition of
$\gamma_X$. It therefore suffices to prove $(a)$. The existence of $u$
and $v$ are direct consequences of the basic properties of almost
split sequences; also, for every $i$, the given $n_i$-tuples are indeed
bases because they arise from  a left (or right) minimal  almost split
morphism (see \ref{sec:remind-ausl-reit}). Let $w\colon \oplus_iX_i^{n_i}\to
\oplus_iX_i^{n_i}$ be $v^{-1}$. For every
$i,i'\in\{1,\ldots,r\}$, and every $1\leqslant j\leqslant n_i$, and every
$1\leqslant j'\leqslant n_{i'}$, let $v_{(i,j),(i',j')}\colon X_i\to
X_{i'}$ and $w_{(i,j),(i',j')}\colon X_i\to
X_{i'}$ be the respective components of $v$ and $w$, from the $j$-th
component $X_i$ of $X_i^{n_i}$ to the $j'$-th component $X_{i'}$ of
$X_{i'}^{n_{i'}}$. Then, the equality $wv=\mathrm{Id}$ reads
\begin{equation}
\label{eq_hg}
  \forall (i,j),(i'',j'')\ \ \
  \sum\limits_{(i',j')}w_{(i,j),(i',j')}v_{(i',j'),(i'',j'')}=\left\{
    \begin{array}{ll}
      \mathrm{Id}_{X_i} &\text{if $(i,j)=(i'',j'')$}\\
      0 & \text{otherwise.}
    \end{array}\right.
\end{equation}
Now, consider the equalities $uf=\alpha v$ in
$\hom_A(\tau_A X,\oplus_i X_i^{n_i})$ and $g=(wvg=)w\beta$ in
$\hom_A(\oplus_iX^{n_i},X)$. Given $i\in \{1,\ldots,r\}$ and $j\in
\{1,\ldots,n_i\}$, the composition with the $(i,j)$-th canonical section
$X_i\to X_i^{n_i}\to \oplus_{i'}X_{i'}^{n_{i'}}$ and with the $(i,j)$-th
canonical retraction $\oplus_{i'}X_{i'}^{n_{i'}}\to X_i^{n_i}\to X_i$,
respectively yield that the following diagrams are commutative,
\[
  \xymatrix@R=15pt{
\tau_A X \ar[rr]^\alpha \ar[d]_{uf_{i,j}} &&
\bigoplus\limits_{i'}X_{i'}^{n_{i'}} \ar[lld]^{
\left[
v_{(i',j'),(i,j)}\,;\,(i',j')
\right]^t}\\
X_i
}
\hskip 1cm \text{and}
\hskip 1cm
\xymatrix{
&& X_i \ar[lld]_{
\left[
w_{(i,j),(i',j')}\,;\,(i',j')
\right]
} \ar[d]^{g_{i,j}}\\
\bigoplus\limits_{i'}X_{i'}^{n_{i'}} \ar[rr]_{\beta} && X &.
}
\]
Thus
$\sigma_*\left(\overline{g_{i,j}}\otimes
    \overline{uf_{i'',j''}}\right) = \sum\limits_{i',j'}
  w_{(i,j),(i',j')}v_{(i',j'),(i'',j'')}$ for every $(i,j)$ and $(i'',j'')$.
This and (\ref{eq_hg}) show $(a)$.
\end{proof}

\subsection{Radical in mesh-categories}
\label{subsec_radical}

% MODIFICATION:
%
% 1) the definition of the radical of a mesh-category that is given in
% the introduction is precise enough in order to avoid recalling it
% here; since it is necessary to save space I erased the definition here.
%
% 2) this part is only applied to translation quivers that are
% coverings (with length) of Auslander-Reiten components; also the
% part is not applied to Auslander-Reiten components themselves; so I
% suggest to use \widetilde\Gamma instead of \Gamma for the
% translation quiver here.
%
Let $\widetilde\Gamma$ be a modulated translation quiver.
Recall that the radical $\mathfrak R\k(\widetilde\Gamma)$ of
$\k(\widetilde \Gamma)$ was defined in the introduction.
For every arrow $x\to y$ in $\widetilde\Gamma$ the
natural  map $M(x,y)\to \k(\widetilde\Gamma)(x,y)$ is one-to-one. And
the $\kappa_x-\kappa_y$-bimodule
$\mathfrak R\k(\widetilde\Gamma)(x,y)$ decomposes as $M(x,y)\oplus \mathfrak
R^2\k(\widetilde\Gamma)(x,y)$.
% MODIFICATION: the definition of translation quiver "with length" is
% already given in the introduction, hence it is erased here.
The description of the ideal $\mathfrak R^\ell\k(\widetilde\Gamma)$ is easier when
$\widetilde\Gamma$ is with length as shows the following
proposition. It is central in this text and
used without further reference. The proof is a small variation of
\cite[2.1]{MR2586985} where the description was first proved in the
case $\kappa_x=\k$ for every vertex $x\in \widetilde\Gamma$.

\begin{prop}
  \label{prop:with_length}
  Let $\widetilde\Gamma$ be a translation quiver with length and $x,y\in\widetilde\Gamma$. If
  there is a path of length $\ell$ from $x$ to $y$ in $\widetilde\Gamma$, then:
  \begin{enumerate}[(a)]
  \item
    $\k(\widetilde\Gamma)(x,y) = \mathfrak{R}\k(\widetilde\Gamma)(x,y)
    = \mathfrak{R}^2\k(\widetilde\Gamma)(x,y) = \cdots = \mathfrak{R}^\ell\k(\widetilde\Gamma)(x,y)$. 
  \item $\mathfrak{R}^i\k(\widetilde\Gamma)(x,y)=0$ if $i>\ell$.
  \end{enumerate}
\end{prop}

\subsection{Coverings of translation quivers}
\label{subsec_cover}

See \cite[1.3]{MR643558} for more details.
A \emph{covering of
  translation quivers}  is a quiver morphism $p\colon \widetilde\Gamma\to \Gamma$ 
such that
\begin{enumerate}[(a)]
\item $\Gamma,\widetilde\Gamma$ are  translation quivers, $\Gamma$ is connected,
\item a vertex $x\in\widetilde\Gamma$ is projective (or injective,
  respectively) if and only if so is $px$,
\item $p$ commutes with the translations in $\Gamma$ and $\widetilde\Gamma$
  (where these are defined),
\item for every vertex $x\in\widetilde\Gamma$ the map $\alpha\mapsto
  p(\alpha)$ induces a bijection from the set of arrows in $\widetilde\Gamma$
  starting in $x$ (or ending in $x$) to the set of arrows in $\Gamma$
  starting in $px$ (or ending in $px$, respectively).
\end{enumerate}
Let $\pi\colon \widetilde\Gamma\to \Gamma$ be a covering of
translation quivers.  If $\Gamma$ is modulated by division
$\k$-algebras $\kappa_x$ and bimodules $M(x,y)$ for every vertex $x$
and every arrow $x\to y$, then $\widetilde\Gamma$ is modulated by the
division $\k$-algebra $\kappa_x:=\kappa_{\pi x}$ at the vertex
$x\in\widetilde\Gamma$ and by the bimodule $M(\pi x,\pi y)$ for every
arrow $x\to y$ in $\widetilde\Gamma$. In this text, this modulation on
$\widetilde\Gamma$ is called \emph{induced} by the modulation
of $\Gamma$.  When $\widetilde\Gamma$ is with length, it has a
\emph{length function,} that is, a map $x\mapsto \ell(x)$ defined on
vertices such that $\ell(y)=\ell(x)+1$ for every arrow $x\to y$. See
\cite[1.6]{MR643558} for the construction of such a function in the
particular case where $\widetilde\Gamma$ is the universal cover of
$\Gamma$ (\cite[1.2,1.3]{MR643558}). Note that quivers have no
parallel arrows here hence the notion of universal cover coincides
with that of \emph{generic cover} used in \cite{MR2819689}).
In the rest of the text, whenever $\Gamma$ is an Auslander-Reiten
component of $A$ and $\widetilde\Gamma$ is its universal cover, the
modulation of $\widetilde\Gamma$ induced by the standard modulation of
$\Gamma$ is referred to as the \emph{universal} modulation.

\medskip

From now on, $A$ is a finite-dimensional $\k$-algebra. Its
Auslander-Reiten components and  their 
coverings are modulated as in \ref{subsec_modulation} and
above, respectively.
To avoid possible confusions, upper-case letters
($X,Y,\ldots$) stand for vertices in Auslander-Reiten components  and lower-case letters ($x,y,\ldots$)
stand for vertices in other translation quivers. But the same notation
($\kappa,M$) is used for all $\k$-modulations.

\section{Well-behaved functors}
\label{sec:wb}

Recall the convention set in the introduction: whenever
$f \colon X \to Y$ and $g \colon Y \to Z$ are two mappings, their
composition is denoted by $fg \colon X \to Z$. Let $\Gamma$ be an
Auslander-Reiten component of $A$ and
$\pi\colon \widetilde\Gamma\to \Gamma$ be a covering of translation
quivers such that $\widetilde\Gamma$ is connected.  This section
introduces the notion of well-behaved functors
$F\colon \k(\widetilde\Gamma)\to \mathrm{ind}\,\Gamma$ where
$\widetilde \Gamma$ is equipped with the modulation induced from the
standard one restricted to $\Gamma$ (see
\ref{subsec_modulation}). These functors are proved to exist when
$\widetilde\Gamma$ is with length.  Such objects were first introduced
over algebraically closed fields for Auslander-Reiten quivers of
algebras of finite representation type (\cite[Sect. 1]{MR576602} and
\cite[Sect. 2]{MR643558}). The section also proves the lifting
properties of well-behaved functors
% MODIFICATION: THE PARENTHESIS IS MODIFIED TO SPARE SPACE AND BECAUSE
% THE SAME IS TOLD 2.6
(see \cite[Sect. 2]{MR2819689}
when $\k$ is algebraically closed).

Until the end of \ref{subsec_wellbehavedexistence} the following
convention is used implicitely. If
$f\colon X\to \oplus_{i=1}^rX_i^{n_i}$ is  an irreducible morphism it is assumed that
$X\in \Gamma$, that $X_1,\ldots,X_r\in \Gamma$ are pairwise distinct,
and $n_1,\ldots,n_r\geqslant 1$. Then $f$ is written
$f=[f_{i,j}\,;\,i,j]$.

\subsection{Sections of residue fields and of spaces of irreducible
  morphisms}

Let $X\in\Gamma$ be a vertex.
The following theorem plays a central role in this article.
\begin{thm}[{Wedderburn, Malcev (see \cite[11.6 Corollary on p. 211]{MR674652})}]
  Let $\k$ be a perfect field. Let $E$ be a finite dimensional
  $\k$-algebra. Then, there exists a section $E/\rad E\to E$ of the
  $\k$-algebra canonical surjection $E\to E/\rad E$.
\end{thm}
For a given section $\kappa_X=\mathrm{End}_A(X)/\rad(X,X)\to \mathrm{End}_A(X)$ of the $\k$-algebra
surjection $\mathrm{End}_A(X)\to \mathrm{End}_A(X)/\rad(X,X)$, the image is denoted by $\k_X$. Then
$\k_X\subseteq \mathrm{End}_A(X)$ is a subalgebra such that $\mathrm{End}_A(X)=\k_X\oplus \rad(X,X)$ as a $\k$-vector space and the
surjection $\mathrm{End}_A(X)\twoheadrightarrow \kappa_X$ restricts to a
$\k$-algebra isomorphism $\k_X\xrightarrow{\sim} \kappa_X$. 
For
short, $\k_X$ is called a \emph{section of $\kappa_X$}. 

Let $X\to Y$ be an arrow in $\Gamma$. Then $X\not\simeq Y$ and $\mathrm{Hom}_A(X,Y)=\rad(X,Y)$. Suppose given sections $\k_X\subseteq
\mathrm{End}_A(X)$ and $\k_Y\subseteq \mathrm{End}_A(Y)$ of $\kappa_X$ and
$\kappa_Y$, respectively. Then $\mathrm{irr}(X,Y)$ is a
$\k_X-\k_Y$-bimodule using the isomorphisms
$\k_X\xrightarrow{\sim}\kappa_X$ and $\k_Y\xrightarrow{\sim}
\kappa_Y$. By a \emph{$\k_X-\k_Y$-linear section}
 (of $\mathrm{irr}(X,Y)$) is meant a section $\mathrm{irr}(X,Y) \to \rad(X,Y)$ of the canonical surjection $\rad(X,Y)\twoheadrightarrow \mathrm{irr}(X,Y)$ in the category of
 $\k_X-\k_Y$-bimodules. Such a section always exists because the
 $\k$-algebra $\k_X\otimes_{\k}\k_Y^{op}$ is
 semisimple. Note that the datum of a linear section 
 depends on the choice of the algebra sections $\k_X$ and $\k_Y$.

 \subsection{Well-behaved functors}\label{subsec_defwellbehaved}
 % MODIFICATION: references added for covering functors.
 The following definition extends to the case of perfect fields the
 already known definition of well-behaved functors when the base field
 is algebraically closed (\cite[3.1]{MR643558} and
 \cite[2.2]{MR576602}).  Keep the setting established at the beginning of this
 section. A $\k$-linear functor
 $F\colon \k(\widetilde\Gamma)\to \mathrm{ind}\,\Gamma$ is called
 \emph{well-behaved} if
  \begin{enumerate}[(a)]
\item $Fx=\pi x$ for every vertex $x\in\widetilde\Gamma$,
  \item for every vertex $x\in \widetilde\Gamma$, the $\k$-algebra map
    from $\kappa_{\pi x}=\kappa_x$ to $\mathrm{End}_A(\pi x)$ defined by $u\mapsto F(u)$
is a section of the natural surjection $\mathrm{End}_A(\pi
x)\twoheadrightarrow \kappa_{\pi x}$. Its image is denoted by $\k_x$,
\item for every arrow $x\to y$ in $\widetilde \Gamma$, the following $\k$-linear
  composite map is a $\k_x-\k_y$-linear section in the obvious way
  \[
    \mathrm{irr}(\pi x,\pi y)= M(x,y)\xrightarrow{\id}
    \k(\widetilde \Gamma)(x,y)\xrightarrow{F}\mathrm{Hom}_A(\pi x,\pi y)\,.
    \]
  \end{enumerate}

Note that if $F$ is as in the definition then distinct vertices
$x,x'\in \widetilde\Gamma$ such that $\pi x=\pi x'$ may give rise to
different sections $\k_x$ and $\k_{x'}$ in $\mathrm{End}_A(\pi x)$.
The data of a section $\k_x\subseteq \mathrm{End}_A(\pi x)$ of
$\kappa_{\pi x}$, for every vertex $x\in\widetilde\Gamma$, and that of
a $\k_x-\k_y$-linear section $M(x,y)\to \mathrm{Hom}_A(\pi
x,\pi y)$, for every arrow $x\to y$ in $\widetilde \Gamma$ determine a
unique $\k$-linear functor $\k\widetilde\Gamma\to \mathrm{ind}\,\Gamma$. It induces a $\k$-linear functor $F\colon
\k(\widetilde \Gamma)\to \mathrm{ind}\,\Gamma$ if and only if it vanishes
on $\gamma_x=\gamma_{\pi x}$ for every non-projective vertex
$x\in\widetilde \Gamma$. In such a case, $F$ is
well-behaved. Moreover, any well-behaved functor arises in this way.

\subsection{Local sections on almost split sequences}
 \label{subsec_localsection}
The existence of well-behaved functors is based on the following technical lemma. It aims
at constructing sections that are compatible with the standard modulation of
$\Gamma$, in some sense.

\begin{lem}
  Let $X\in\Gamma$ be a non-injective vertex and let
  \[
        0\to X\xrightarrow{f}
    \bigoplus_{i=1}^rX_i^{n_i}\xrightarrow{
g} \tau_A^{-1}X\to 0
\]
be an almost split sequence. Let $\k_X\subseteq
\mathrm{End}_A(X)$ and $\k_{X_i}\subseteq \mathrm{End}_A(X_i)$ (for every
$i\in\{1,\ldots,r\})$ be sections of $\kappa_X$ and $\kappa_{X_i}$
respectively, and let $\mathrm{irr}(X,X_i)\hookrightarrow \rad(X,X_i)$ be a
$\k_X-\k_{X_i}$-linear section which maps
$\overline{f_{i,1}},\ldots,\overline{f_{i,n_i}}$to
$f_{i,1},\ldots,f_{i,n_i}$ respectively, for every $i\in\{1,\ldots,r\}$. There
  exists a section $\k_{\tau_A^{-1}X}\hookrightarrow \mathrm{End}_A(\tau_A^{-1}X)$ of $\kappa_{\tau_A^{-1}X}$ and a
  $\k_{X_i}-\k_{\tau_A^{-1}X}$-linear section $\mathrm{irr}(X_i,\tau_A^{-1}X)\to \rad(X_i,\tau_A^{-1}X)$ (for every
  $i\in\{1,\ldots,r\}$) such that
  \begin{enumerate}[(a)]
  \item it maps $\overline{g_{i,1}},\ldots,\overline{g_{i,n_i}}$ to
$g_{i,1},\ldots,g_{i,n_i}$ respectively, for every $i\in\{1,\ldots,r\}$,
\item the induced map
  $\oplus_{i=1}^r\mathrm{irr}(X,X_i)\otimes_{\kappa_{X_i}}
  \mathrm{irr}(X_i,\tau_A^{-1}X)\to \mathrm{Hom}_A(X,\tau_A^{-1}X)$
  vanishes on $\gamma_{\tau_A^{-1}X}$.
  \end{enumerate}
\end{lem}
\begin{proof}
  Note that if such sections do exist then $(g_{i,1},\ldots,g_{i,n_i})$
  must be a basis over $\k_{X_i}$ of the image of $\mathrm{irr}(X_i,\tau_A^{-1}X)\to \rad(X_i,\tau_A^{-1}X)$ because
  $g$ is a right minimal almost split morphism. In particular, the
  section $\k_{\tau_A^{-1}X}\subseteq \mathrm{End}_A(\tau_A^{-1}X)$ must
  be such that the left $\k_{X_i}$-submodule of $\rad(X_i,\tau_A^{-1}X)$ generated by $g_{i,1},\ldots,g_{i,n_i}$ is
  a $\k_{X_i}-\k_{\tau_A^{-1}X}$-submodule of $\rad(X_i,\tau_A^{-1}X)$, for every $i\in\{1,\ldots,r\}$. The proof
  therefore proceeds as follows: 1) define a section
  $\k_{\tau_A^{-1}X}\subseteq\mathrm{End}_A(\tau_A^{-1}X)$ satisfying
  this last condition, 2) define sections $\mathrm{irr}(X_i,\tau_A^{-1}X)\to \rad(X_i,\tau_A^{-1}X)$ so that $(a)$
  is satisfied, and 3) prove $(b)$.

1) Let $\varphi\in \mathrm{End}_A(\tau_A^{-1}X)$ and define a new
representative $\varphi_1\in \mathrm{End}_A(\tau_A^{-1}X)$ of
$\overline\varphi\in\kappa_{\tau_A^{-1}X}$ as
follows. Let $0\to
X\xrightarrow{\alpha}\oplus_iX_i^{n_i}\xrightarrow{\beta}\tau_A^{-1}X\to
0$ be the almost split sequence used in the definition of the
standard modulation of $\Gamma$. So there exists an isomorphism of exact
sequences
\begin{equation}
\label{eq_isoexact}
    \xymatrix@R=15pt{
0 \ar[r] & X\ar[r]^\alpha \ar[d]_u &
\bigoplus_iX_i^{n_i}\ar[r]^\beta \ar[d]^v & \tau_A^{-1}X\ar[r] \ar@{=}[d] & 0\\
0 \ar[r] & X\ar[r]_f &
\bigoplus_iX_i^{n_i}\ar[r]_g & \tau_A^{-1}X\ar[r]  & 0 &.
}
\end{equation}
There also exists a
commutative diagram
\begin{equation}
\label{eq_phi}
   \xymatrix@R=15pt{
0 \ar[r] & X\ar[r]^f\ar[d]_\psi &
\bigoplus_iX_i^{n_i}\ar[r]^g \ar[d]^\theta & \tau_A^{-1}X\ar[r] \ar[d]^\varphi & 0\\
0 \ar[r] & X\ar[r]_f &
\bigoplus_iX_i^{n_i}\ar[r]_g & \tau_A^{-1}X\ar[r]  & 0 &
}
\end{equation}
for some morphisms $\theta$ and $\psi$. The pair $(\theta,\psi)$ in
the above diagram is not unique; however for any different pair
$(\theta',\psi')$ such that the diagram (\ref{eq_phi}) still commutes
after replacing $\theta$ and $\psi$ by $\theta'$ and $\psi'$,
respectively, then there exists $h\colon \oplus_i X^{n_i}\to X$ such
that $\theta-\theta' = hf$; given that
$(\psi-\psi')f = f(\theta - \theta')$ and that $f$ is a monomorphism,
it follows that $\psi-\psi'= fh$; note that none of the $X_i$ is isomorphic
to $X$ because $f$ is irreducible; therefore $h$ lies in $\rad$, and
hence $\psi-\psi'$ lies in $\rad^2$. The diagrams (\ref{eq_isoexact})
and (\ref{eq_phi}) yield the following commutative diagram
\[
      \xymatrix@R=15pt{
0 \ar[r] & X\ar[r]^\alpha \ar[d]_{u\psi u^{-1}} &
\bigoplus_iX_i^{n_i}\ar[r]^\beta \ar[d]^{v\theta v^{-1}}& \tau_A^{-1}X\ar[r] \ar[d]^\varphi & 0\\
0 \ar[r] & X\ar[r]_\alpha &
\bigoplus_iX_i^{n_i}\ar[r]_\beta & \tau_A^{-1}X\ar[r]  & 0 &.
}
\]
Therefore,
$\tau_*(\overline \varphi)=\overline{u\psi u^{-1}}\in\kappa_X$. The
previous discussion on the uniqueness of the pair $(\theta,\psi)$
shows that $\overline{\psi}$ is uniquely determined by $\varphi$ even
though $\psi$ is not. Now let $\psi_1\in \k_X$ be the representative
of $\psi$, that is, $\overline{\psi_1}=\overline{\psi}$.  Since the
section $\mathrm{irr}(X,X_i)\to \rad(X,X_i)$ is
$\k_X-\k_{X_i}$-linear and maps
$\overline{f_{i,1}},\ldots,\overline{f_{i,n_i}}$ to
$f_{i,1},\ldots,f_{i,n_i}$, respectively, and since
$(\overline{f_{i,1}},\ldots,\overline{f_{i,n_i}})$ is a basis of the
right $\k_{X_i}$-module $\mathrm{irr}(X,X_i)$, there is a unique
matrix $\eta_i\in M_{n_i}(\k_{X_i})$, considered as an endomorphism of
$X_i^{n_i}$, making the following square commute for
$i\in\{1,\ldots,r\}$
\[
  \xymatrix@R=15pt{
X \ar[rr]^{[f_{i,1},\ldots,f_{i,n_i}]} \ar[d]_{\psi_1} && X_i^{n_i} \ar[d]^{\eta_i} \\
X \ar[rr]_{[f_{i,1},\ldots,f_{i,n_i}]} && X_i^{n_i}& .
}
\]
Therefore there exists a unique
$\varphi_1\in \mathrm{End}_A(\tau_A^{-1}X)$ making the following diagram
commute where $\eta\colon \oplus_i X_i^{n_i}\to \oplus_i X_i^{n_i}$
is the morphism defined by $\{\eta_i\colon X_i^{n_i}\to X_i^{n_i}\}_{i}$.
\begin{equation}
\label{eq_phi1}
    \xymatrix@R=15pt{
0 \ar[r] & X\ar[r]^f \ar[d]_{\psi_1}&
\bigoplus_iX_i^{n_i}\ar[r]^g \ar[d]^\eta &
\tau_A^{-1}X\ar[r] \ar[d]^{\varphi_1} & 0\\
0 \ar[r] & X\ar[r]_f &
\bigoplus_iX_i^{n_i}\ar[r]_g & \tau_A^{-1}X\ar[r]  & 0 &.
} 
\end{equation}
Hence the following diagram commutes
\[
      \xymatrix@R=15pt{
0 \ar[r] & X\ar[r]^\alpha \ar[d]_{u\psi_1u^{-1}} &
\bigoplus_iX_i^{n_i}\ar[r]^\beta \ar[d]^{v \eta v^{-1}} & \tau_A^{-1}X\ar[r] \ar[d]^{\varphi_1} & 0\\
0 \ar[r] & X\ar[r]_\alpha &
\bigoplus_iX_i^{n_i}\ar[r]_\beta & \tau_A^{-1}X\ar[r]  & 0 &.
}
\]
This entails
$\overline{\varphi_1}=\tau_*^{-1}(\overline{u\psi_1u^{-1}})$. But
  $\overline{\psi_1}=\overline{\psi}$ so that
  $\overline{u\psi_1u^{-1}}=\overline{u\psi u^{-1}}$. Thus
  $\overline{\varphi_1}=\tau_*^{-1}(\overline{u\psi
    u^{-1}})=\overline{\varphi}$. This construction therefore yields a
  well-defined map
  \[
    \begin{array}{crcl}
    s\colon  &\mathrm{End}_A(\tau_A^{-1}X) & \to  & \mathrm{End}_A(\tau_A^{-1}X)\\
     & \varphi & \mapsto &\varphi_1\ \ \text{(such that
       $\overline{\varphi_1}=\overline \varphi$).}
    \end{array}
    \]
Since $\tau_*\colon \kappa_{\tau_A^{-1}X}\to \kappa_X$ is a
$\k$-algebra isomorphism and  $\eta_1,\ldots,\eta_r$ are
determined by $\psi_1$ (and the fixed $\k$-algebra sections),
the map $s$ is a $\k$-algebra homomorphism. Moreover if $\varphi\in
\rad(\tau_A^{-1}X,\tau_A^{-1}X)$ then $\overline{\varphi}=0$ and
the representative $\psi_1$ in $\k_X$ of $\tau_*(\overline\varphi)=0$
is $0$; therefore, $\eta_1,\ldots,\eta_r=0$ and $\varphi_1=0$; in
other words $s$ vanishes on $\rad(\tau_A^{-1}X,\tau_A^{-1}X)$. 
Hence, $s$ induces a $\k$-algebra homomorphism
\[
\overline s \colon \kappa_{\tau_A^{-1}X}\to \End_A(\tau_A^{-1}X)
\]
such that $\overline s(\overline \varphi) = \varphi$ for all
$\varphi \in \End_A(\tau_A^{-1}X)$. Denote by $\k_{\tau_A^{-1}X}$ the
image of $\overline s$. Thus, $\k_{\tau_A^{-1}X}
\subseteq \End_A(\tau_A^{-1}X)$ is a $\k$-algebra section of $\kappa_{\tau_A^{-1}X}$.
% Finally if $\varphi'\in\mathrm{End}_A(\tau_A^{-1}X)$ lies in the
% image of $s$ then $\varphi'=s(\varphi)=\varphi_1$ for some $\varphi\in\mathrm{End}_A(\tau_A^{-1}X)$; then, keeping the same notation for
% the morphisms ($\theta,\psi,\eta_1,\ldots,\eta_r, \eta$) used to define
% $s(\varphi)$ and using dashed notation for the corresponding
% morphisms used to define $s(\varphi')$, it follows that
% $\theta'=\eta$ and $\psi'=\psi_1\in \k_X$; hence
% the representative $\psi'_1\in \k_X$ of $\overline{\psi'}$ is $\psi_1$
% itself so that
% $(\eta_1',\ldots,\eta_r')=(\eta_1,\ldots,\eta_r)$; and thus
% $\varphi_1'=\varphi_1=\varphi'$, that is
% $s(\varphi')=\varphi$. Therefore, $\k_{\tau_A^{-1}X}\subseteq \mathrm{End}_A(\tau_A^{-1}X)$ is a $\k$-algebra section of
%   $\kappa_{\tau_A^{-1}X}$ where  $\k_{\tau_A^{-1}X}$
%   denotes the image of $s$.

2) In order to prove $(a)$ there only remains to define a
$\k_{X_i}-\k_{\tau_A^{-1}X}$-linear section
$\mathrm{irr}(X_i,\tau_A^{-1}X)\to \rad(X_i,\tau_A^{-1}X)$
which maps $\overline{g_{i,1}},\ldots,\overline{g_{i,n_i}}$ to
$g_{i,1},\ldots,g_{i,n_i}$, respectively, for every
$i\in\{1,\ldots,r\}$.  Let $i\in\{1,\ldots,r\}$. Since $g$ is right
minimal almost split, then
$\{\overline{g_{i,j}}\}_{1\leqslant j \leqslant n_i}$ is a basis of
$\irr(X_i,\tau_A^{-1}X)$ as a left $\kappa_{X_i}$-vector space, and
hence as a left $\k_{X_i}$-vector space. Denote by $s_i$ the unique
 $\k_{X_i}$-linear mapping 
\[
s_i \colon \irr(X_i,\tau_A^{-1}X) \to \rad(X_i,\tau_A^{-1}X)
\]
which maps $\overline{g_{i,j}}$ to $g_{i,j}$ for every
$j\in \{1,\ldots,n_i\}$. In view of the assumption on the section
$\k_{X_i}\subseteq \End_A(X_i)$ of $\kappa_{X_i}$, the construction of
$s_i$ entails that the identity mapping of $\irr(X_i,\tau_A^{-1}X)$ is
equal to the composite morphism
\[
\irr(X_i,\tau_A^{-1}X) \xrightarrow{s_i} \rad(X_i,\tau_A^{-1}X)
\twoheadrightarrow \irr(X_i,\tau_A^{-1}X)\,,
\]
where the right hand-side arrow is the natural surjection. In order to
prove that $s_i$ is a $\k_{X_i}-\k_{\tau_A^{-1}X}$-linear section,
there only remains to prove that $s_i$ is
$\k_{\tau_A^{-1}X}$-linear. Let $\varphi \in \k_{\tau_A^{-1}X}$. Using
the above notation in the construction of $\k_{\tau_A^{-1}X}$, one has
$\varphi=\varphi_1$.  It follows from (\ref{eq_phi1}) that
\[
[g_{i,1},\ldots,g_{i,n_i}]^t\cdot\varphi=\eta_i\cdot
[g_{i,1},\ldots,g_{i,n_i1}]^t\,.
\]
Writing the endomorphism $\eta_i\colon X_i^{n_i} \to X_i^{n_i}$
coordinatewise as $[\eta_{i,j,j'}\,;\,1\leqslant j,j'\leqslant n_i]$,
where $\eta_{i,j,j'}\in \k_{X_i}$ for all $j,j'$, it follows that
\[
(\forall j\in \{1,\ldots,n_i\})\ g_{i,j}\varphi = \sum_{j'=1}^{n_i}
\eta_{i,j,j'} g_{i,j'}\,.
\]
Applying $s_i$ which is $\k_{X_i}$-linear and satisfies
$s_i(\overline{g_{i,j}}) = g_{i,j}$ for all $j$ entails that
% Since $s_i$ is $\k_{X_i}$-linear and $s_i(\overline{g_{i,j}}) =
% g_{i,j}$ for all $j$, this entails that
\[
(\forall j\in \{1,\ldots,n_i\})\ 
s_i(\overline{g_{i,j}}\overline\varphi) =
% \sum_{j'}\eta_{i,j,j'} s_i(\overline{g_{i,j'}}) =
% \sum_{j'}\eta_{i,j,j'} g_{i,j'} = 
g_{i,j}\varphi =
s_i(\overline{g_{i,j}}) \varphi\,.
\]
Accordingly, $s_i$ is $\k_{\tau_A^{-1}X}$-linear, which finishes the
proof of $(a)$.

3) There only remains to prove (b). According to the lemma in
\ref{subsec_modulation}, the commutative diagram (\ref{eq_isoexact})
entails that
% MODIFICATION: THE FORMULA IS NOT IN "DISPLAYMATH" ANYMORE TO SPARE SPACE"
$\gamma_{\tau_A^{-1}X}=\sum\limits_{i=1}^r\sum\limits_{j=1}^{n_i}\overline
u\overline{f_{i,j}}\otimes \overline{g_{i,j}}$.
Let $u'\in \k_{X}$ be the representative of
$\overline u\in \kappa_{X}$. The image of $\gamma_{\tau_A^{-1}X}$
under the map
$\oplus_{i=1}^r\mathrm{irr}(X,X_i)\otimes_{\kappa_{X_i}}
\mathrm{irr}(X_i,\tau_A^{-1}X)\to \mathrm{Hom}_A(X,\tau_A^{-1}X)$
induced by all the considered sections is therefore
$\sum\limits_{i,j}u'\,f_{i,j}\,g_{i,j}=u'\sum\limits_{i,j}f_{i,j}\,g_{i,j}=u'fg=0$.
\end{proof}

Dual considerations yield the following dual version of the
preceding lemma.
\begin{lem}
  Let $X\in\Gamma$ be a non-injective vertex and let
  \[
        0\to X\xrightarrow{f}
    \bigoplus_{i=1}^rX_i^{n_i}\xrightarrow{
g} \tau_A^{-1}X\to 0
\]
be an almost split sequence. Let
$\k_{\tau_A^{-1}X}\subseteq \mathrm{End}_A(X)$ and
$\k_{X_i}\subseteq \mathrm{End}_A(X_i)$ (for
$i\in\{1,\ldots,r\})$ be sections of $\kappa_{\tau_A^{-1}X}$ and
$\kappa_{X_i}$ respectively, and let
$\mathrm{irr}(X_i,\tau_A^{-1}X)\hookrightarrow
\rad(X_i,\tau_A^{-1}X)$
be a $\k_{X_i}-\k_{\tau_A^{-1}X}$-linear section under which
$\overline{g_{i,1}},\ldots,\overline{g_{i,n_i}}$ are mapped to
$g_{i,1},\ldots,g_{i,n_i}$ respectively, for every
$i\in\{1,\ldots,r\}$. There exists a section
$\k_X\hookrightarrow \mathrm{End}_A(X)$ of
$\kappa_{X}$ and a $\k_{X}-\k_{X_i}$-linear
section
$\mathrm{irr}(X,X_i)\to \rad(X,X_i)$
(for every $i\in\{1,\ldots,r\}$) such that
  \begin{enumerate}[(a)]
  \item it maps $\overline{f_{i,1}},\ldots,\overline{f_{i,n_i}}$ to
$f_{i,1},\ldots,f_{i,n_i}$ respectively, for every $i\in\{1,\ldots,r\}$,
\item the induced map
  $\oplus_{i=1}^r\mathrm{irr}(X,X_i)\otimes_{\kappa_{X_i}}
  \mathrm{irr}(X_i,\tau_A^{-1}X)\to \mathrm{Hom}_A(X,\tau_A^{-1}X)$
  vanishes on $\gamma_{\tau_A^{-1}X}$.
  \end{enumerate}
\end{lem}

\subsection{Inductive construction of well-behaved
  functors}\label{subsec_inductivewellbehaved}
In view of proving the existence of well-behaved functors
$\k(\widetilde{\Gamma})\to\mathrm{ind}\,\Gamma$ it is necessary to
consider $\k$-linear functors $F\colon \k\chi\to \mathrm{ind}\,\Gamma$
where $\chi$ is a full and convex subquiver of
$\widetilde\Gamma$. Here the notation $\k \chi$ stands for the full
subcategory of $\k\widetilde\Gamma$ with object set being the set of
vertices in $\chi$. Following \ref{subsec_defwellbehaved}, such a
functor $F\colon \k\chi\to\mathrm{ind}\,\Gamma$ is called well-behaved
if: (a) $Fx=\pi x$ for every vertex $x\in\chi$; (b) the $\k$-algebra
homomorphism
$\kappa_x\hookrightarrow \k\chi(x,x)\to \mathrm{End}_A(\pi x)$ is a
section of the natural surjection
$\mathrm{End}_A(\pi x)\twoheadrightarrow \kappa_{\pi x}=\kappa_x$, for
every vertex $x\in\chi$ (as above, the image of the section is denoted
by $\k_x$); (c) the $\k$-linear composite map
$\mathrm{irr}(\pi x,\pi y)= M(x,y)\hookrightarrow \k \chi(x,y)\to
\mathrm{Hom}_A(\pi x,\pi y)$
is a $\k_x-\k_y$-linear section, for every arrow $x\to y$ in $\chi$;
and (d) it vanishes on $\gamma_x=\gamma_{\pi x}$ for every
non-projective vertex $x\in\chi$ such that $\tau x\in \chi$.

% In the following lemma, denote by $\sim$ the
% thinest equivalence relation on the arrows of $\widetilde\Gamma$ such
% that the two following properties hold.
% \begin{itemize}
% \item $\alpha \sim \beta$ for every subquiver of $\widetilde\Gamma$ of
%   the shape
%   $\def\labelstyle{\scriptscriptstyle} \vcenter{\xymatrix@R=1pt{
%       & \circ\\
%       \circ \ar[ru]^\alpha \ar[rd]_\beta \\
%       & \circ }}$, and
% \item $\gamma \sim \delta$ for every subquiver of $\widetilde\Gamma$ of
%   the shape
%   $\def\labelstyle{\scriptscriptstyle} \vcenter{\xymatrix@R=1pt{
%     \circ \ar[rd]^\gamma \\
%     & \circ \\
%     \circ \ar[ru]_\beta \\
% }}$.
% \end{itemize}
% When $\widetilde\Gamma$ is with length and some length function $\ell$
% is given on the vertices of $\widetilde\Gamma$, here is a description
% of the equivalence classes of $\sim$. Let
% $n \in \mathbb Z$, consider the full subquiver of $\widetilde\Gamma$
% with set of vertices equal to
% \[
% \{x\in \widetilde \Gamma\ |\ \ell(x) \in \{n,n+1\}\}\,,
% \]
% and let $C$ be a connected component of it; then, the set of arrows of
% $C$ is an equivalence class for $\sim$. Conversely, any equivalence
% class for $\sim$ is of that shape for some $n \in \mathbb Z$ and some
% $C$. No distinction is made between an equivalence class for $\sim$
% and the associated subquiver of $\widetilde\Gamma$.
\begin{lem}
  Assume that $\widetilde\Gamma$ is with length. Let $\ell$ be a
  length function on the vertices in $\widetilde\Gamma$.
  Let $F\colon \k\chi\to \mathrm{ind}\,\Gamma$ be a well-behaved functor
  where $\chi\subseteq \widetilde\Gamma$ is a full and convex
  subquiver distinct from $\widetilde\Gamma$ and satisfying the
  following two conditions
  \begin{enumerate}[(i)]
  \item either there is no arrow $x\to x_1$ in $\widetilde\Gamma$
    such that $x_1\in\chi$ and $x\not\in\chi$, or else there is an
    upper bound on the integers $\ell(x)$ where $x$ runs through the
    vertices in $\widetilde\Gamma\backslash\chi$ such that there
    exists an arrow $x\to x_1$ in $\widetilde\Gamma$ satisfying
    $x_1\in\chi$,
  \item either there is no arrow $x_1\to x$ in $\widetilde\Gamma$
    such that $x_1\in\chi$ and $x\not\in\chi$, or else there is a
    lower bound on the integers $\ell(x)$ where $x$ runs through the
    vertices in $\widetilde\Gamma\backslash\chi$ such that there
    exists an arrow $x_1\to x$ in $\widetilde\Gamma$ satisfying
    $x_1\in\chi$.
  \end{enumerate}
  Then there exists at least one arrow such as in $(i)$ or $(ii)$ and
  for every arrow $x\to x_1$ (or, $x_1\to x$) in $\widetilde\Gamma$
  such that $x\not\in \chi$ and $x_1\in\chi$, with $\ell(x)$ maximal
  (or, minimal, respectively) for these properties, there exists a
  well-behaved functor $F'\colon \k\chi'\to\mathrm{ind}\,\Gamma$ which
  extends $F$, where $\chi'$ denotes the full subquiver of
  $\widetilde\Gamma$ with vertices $x$ and the vertices in $\chi$.
\end{lem}
\begin{proof}
  There exists an arrow $x\to x_1$ or $x_1\to x$ in $\widetilde\Gamma$
  such that $x\not\in\chi$ and $x_1\in\chi$ because
  $\chi\subsetneq\widetilde\Gamma$. Assume the former (the latter is
  dealt with similarly) and choose $x$ so that $\ell(x)$ is
  maximal ($(i)$). Let $\chi'$ be the full
  subquiver of $\widetilde\Gamma$ with vertices $x$ and those of
  $\chi$. Then $\chi'$ is convex in $\widetilde\Gamma$ because
  $\ell(x)$ is maximal. Let $x\to x_1,\ldots,x\to x_r$ be the arrows
  in $\widetilde\Gamma$ starting in $x$ and ending in some vertex in
  $\chi$. Note that if $x$ is non-injective and $\tau^{-1}x\in\chi$
  then these are all the arrows in $\widetilde\Gamma$ starting in $x$
  because  $\ell(x)$ is maximal. In order to extend $F\colon
  \k\chi\to \mathrm{ind}\,\Gamma$ to a
  functor $F'\colon
  \k\chi'\to \mathrm{ind}\,\Gamma$ distinguish two cases according to
  whether $x$ is non-injective and $\tau^{-1}x\in\chi$, or not. For
  every $i\in\{1,\ldots,r\}$ let $\k_i\subseteq \mathrm{End}_A(\pi x_i)$
  be the section $F(\kappa_{x_i})$ of
  $\kappa_{\pi x_i}$.

  Assume first that either $x$ is injective, or else $x$ is
  non-injective and $\tau^{-1}x\not\in \chi$.  Fix any $\k$-algebra
  section $\k_x\subseteq \mathrm{End}_A(\pi x)$ of $\kappa_{\pi x}$. Let
  $i\in\{1,\ldots,r\}$. The $\k$-algebra isomorphisms
  $\kappa_{x_i}\to \k_i$ (induced by $F$) and $\kappa_x\to \k_x$ allow
  one to consider $M(x,x_i)$ as a $\k_x-\k_i$-bimodule. For this
  structure, the quotient map
  $\mathrm{Hom}_A(\pi x,\pi x_i)=\mathrm{rad}(\pi x,\pi
  x_i)\overset{p_i}{\twoheadrightarrow} \mathrm{irr}(\pi x,\pi x_i)$
  is $\k_x-\k_i$-linear. On the other hand the $\k$-algebra
  $\k_x\otimes_{\k}\k_i^{op}$ is semi-simple.  Hence there
  exists a $\k_x-\k_i$-linear section
  $M(x,x_i)=\mathrm{irr}(\pi x,\pi x_i)\xrightarrow{s_i}\mathrm{Hom}_A(\pi
  x,\pi x_i)$
  of $p_i$. The sections $\k_x\subseteq \mathrm{End}_A(\pi x)$ and $s_i$
  ($i\in\{1,\ldots,r\}$) extend $F\colon \k\chi\to \mathrm{ind}\,\Gamma$
  to a $\k$-linear functor $F'\colon \k\chi'\to \mathrm{ind}\,\Gamma$
  satisfying the conditions (a), (b) and (c) in the definition of
  well-behaved functors. Moreover, condition (d) is satisfied for $F'$
  because it is so for $F$ and because, if $y\in\widetilde\Gamma$ is
  non-projective and $y,\tau y\in \chi'$ then $y,\tau y\in\chi$. This
  proves the lemma when either $x$ is injective or else $x$ is
  non-injective and $\tau^{-1}x\not\in \chi$.

Assume now that $x$ is non-injective and $\tau^{-1}x\in\chi$.  The
mesh in $\widetilde\Gamma$ starting in $x$ has the form
\[
  \xymatrix@R=2pt{
& x_1\ar[rd] \ar@{.}[dd] & \\
x \ar[ru] \ar[rd] &  & \tau^{-1}x\,.\\
&x_r\ar[ru]
}
\]
For simplicity let $X=\pi x$ and $X_i=\pi x_i$ for every
$i\in\{1,\ldots,r\}$ so that $\pi \tau^{-1}x=\tau_A^{-1}X$. Let
$n_1,\ldots,n_r\geqslant 1$ be the integers such that there is an
almost split sequence
$0\to X\to \oplus_{i=1}^rX_i^{n_i}\to \tau_A^{-1}X\to 0$. Since
$F\colon \k \chi \to \mathrm{ind}\,\Gamma$ is well-behaved, it induces a
$\k$-algebra section
$\k_{\tau_A^{-1}X}\subseteq \mathrm{End}_A(\tau_A^{-1}X)$ of
$\kappa_{\tau^{-1}x}=\kappa_{\tau_A^{-1}X}$. It also induces a
$\k_i-\k_{\tau_A^{-1}X}$-linear section
$M(x_i,\tau^{-1}x)=\mathrm{irr}(X_i,\tau_A^{-1}X)\to \mathrm{rad}(X_i,\tau_A^{-1}X)$,
for every $i\in\{1,\ldots,r\}$. Let $i\in\{1,\ldots,r\}$, let
$\{g_{i,j}\colon X_i\to \tau_A^{-1}X\}_{1\leqslant j\leqslant n_i}$ be the image
under this $\k_i-\k_{\tau_A^{-1}X}$-linear section of a basis of the
$\k_i$-vector space $M(x_i,\tau^{-1}x)$. Thus, the morphism
$\bigoplus\limits_{i=1}^rX_i^{n_i} \xrightarrow{g:=\left[g_{i,j}\,;\,
    i,j\right]^t} \tau_A^{-1}X$
is right minimal almost split. Let
$f=\left[f_{i,j}\,;\, i,j\right]\colon X\to \oplus_{i=1}^rX_i^{n_i}$
be its kernel. Then \ref{subsec_localsection} yields a $\k$-algebra
section $\k_X\subseteq \mathrm{End}_A(\tau_A^{-1}X)$ of
$\kappa_X$ together with $\k_X-\k_i$-linear
sections
$M(x,x_i) =\mathrm{irr}(X,X_i)\to \mathrm{rad}(X,X_i)$,
for $1\leqslant i \leqslant r$, such that the induced map
\[
  \bigoplus\limits_{i=1}^r
  M(x,x_i)\otimes_{\kappa_{x_i}} 
  M(x_i,\tau^{-1}x)\longrightarrow \mathrm{Hom}_A(X,\tau_A^{-1}X)
\]
vanishes on $\gamma_{\tau_A^{-1}X}$. These new sections clearly extend
$F\colon \k\chi \to \mathrm{ind}\,\Gamma$ to a well-behaved functor
$F'\colon \k\chi'\to \mathrm{ind}\,\Gamma$. Like in the previous case
$F'\colon \k\chi'\to \mathrm{ind}\,\Gamma$ is
well-behaved. Finally $\chi'$ satisfies both conditions $(i)$ and $(ii)$
in the statement of the lemma since $\chi'\backslash \chi$ consists of
one vertex. 
\end{proof}

\subsection{Existence of well-behaved functors}
\label{subsec_wellbehavedexistence}
For a strongly irreducible morphism $f\colon X \to
\oplus_{i=1}^rX_i^{n_i}$, the $\kappa_X-\kappa_{X_i}$-bimodule
$\kappa_X \cdot \overline{f_{i,j}} \cdot \kappa_{X_i}$ is projective
and given by an idempotent of $\kappa_X\otimes \kappa_{X_i}^\op$, for
all $i,j$, hence, the following are equivalent
\begin{itemize}
\item for all $i,j$, the annihilator of $\overline{f_{i,j}}$ over
  $\kappa_{X}\otimes_{\k} \kappa_{X_i}^{\op}$ is trivial,
\item for all $i$, the family $\{\overline{f_{i,j}}\}_{j}$ of
  $\irr(X,X_i)$ is free over
  $\kappa_{X} \otimes_{\k} \kappa_{X_i}^\op$.
\end{itemize}
The following result implies Theorem~\ref{thm1}.  
\begin{prop}
% MODIFICATION : THE UNIVERSAL COVER IS REPLACED BY A TRANSLATION
% QUIVER WITH LENGTH
  Let $A$ be a finite-dimensional algebra over a perfect field
  $\k$. Let $\Gamma$ be an Auslander-Reiten component of $A$. Let
  $\pi\colon\widetilde\Gamma\to \Gamma$ be a covering of translation
  quivers where $\widetilde\Gamma$ is connected and with length. Endow
  $\Gamma$ with its standard $\k$-modulation and $\widetilde \Gamma$
  with the induced $\k$-modulation.
  \begin{enumerate}
  \item There exists a well-behaved functor $F\colon \k(\widetilde
    \Gamma)\to \mathrm{ind}\,\Gamma$.
  \item Let $X\in \Gamma$ and let
    $f\colon X\to \oplus_{i=1}^rX_i^{n_i}$ be a strongly irreducible
    morphism. Assume that the annihilator of $\overline{f_{i,j}}$ over
    $\kappa_X\otimes_\k\kappa_{X_i}^\op$ is zero for all $i,j$. Let
    $x\in\pi^{-1}(X)$ and let
    
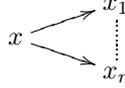
\[
    \xymatrix@R=2pt{
      & x_1 \ar@{.}[dd]\\
      x \ar[ru] \ar[rd]\\
      &x_r
    }
    \]
be the full  subquiver of $\widetilde\Gamma$ such that $\pi
x_i=X_i$. Then there exists a well-behaved functor $F\colon
\k(\widetilde\Gamma)\to \mathrm{ind}\,\Gamma$ such that $F$ maps
$\overline{f_{i,j}}\in \k(\widetilde\Gamma)(x,x_i)$ to $f_{i,j}$ for
every 
$i\in\{1,\ldots,r\}$ and $j\in\{1,\ldots,n_i\}$.
  \end{enumerate}
\end{prop}
\begin{proof}
  It is possible to assume that $A$ is connected and
  not semi-simple. Both (1) and (2) are proved with the same inductive
  construction yet with a specific initialisation step for each. In
  case (1), let $r=1$, let $X\to X_1$ be any arrow in $\Gamma$ and let
  $x\to x_1$ be any arrow in $\widetilde \Gamma$ with image under
  $\pi$ equal to $X\to X_1$. In both cases, let $\chi_0$ be the full
  subquiver of $\widetilde\Gamma$ with vertices $x,x_1,\ldots,x_r$.
  
  Let $\Sigma$ be the set of pairs
  $(\chi,F\colon \k\chi\to \mathrm{ind}\,\Gamma)$ where
  $\chi\subseteq\widetilde \Gamma$ is a full and convex subquiver
  containing $\chi_0$ and which satisfies conditions $(i)$ and $(ii)$
  in \ref{subsec_inductivewellbehaved}, and
  $F\colon \k\chi\to \mathrm{ind}\,\Gamma$ is a well-behaved functor
  with the following additional condition in case (2): It maps
  $\overline{f_{i,j}}\in\k(\widetilde\Gamma)(x,x_i)$ to $f_{i,j}$ for
  all $i,j$. This set is ordered: $(\chi,F)\leqslant (\chi',F')$ if and only
  if $\chi\subseteq \chi'$ and $F'$ restricts to $F$.
  
  Let $\ell$ be a length function on $\widetilde\Gamma$. Denote
  $\ell(x)$ by $\ell_0$. Denote by $\widetilde\Gamma_{\geqslant
    \ell_0}$ the full subquiver of $\widetilde\Gamma$ with set of
  vertices being $\{x'\in \widetilde\Gamma\ |\ \ell(x')\geqslant
  \ell_0\}$. Consider the subsets $\Sigma_1$, $\Sigma_2$ and
  $\Sigma_3$ of $\Sigma$ defined as follows.
  \begin{itemize}
  \item $\Sigma_1$ consists of those $(\chi,F)\in \Sigma$ such that
    the restriction of $\ell$ to $\chi$ is bounded below by $\ell_0$
    and bounded above by $\ell_0+1$.
  \item $\Sigma_2$ consists of those $(\chi,F)\in \Sigma$ such that
    $\chi \subseteq \widetilde\Gamma_{\geqslant \ell_0}$ and $\chi$ is
    stable under predecessors in $\widetilde\Gamma_{\geqslant
      \ell_0}$, that is, any oriented path in
    $\widetilde\Gamma_{\geqslant \ell_0}$ with endpoint in $\chi$ is
    contained in $\chi$.
  \item $\Sigma_3$ consists of the those $(\chi,F)\in \Sigma$ such
    that $\widetilde\Gamma_{\geqslant \ell_0}\subseteq \chi$ and
    $\chi$ is stable under successors in $\widetilde\Gamma$.
  \end{itemize}
  
  The set $\Sigma_1$ is not empty. Indeed, let
  $F\colon \k \chi_0\to \ind\Gamma$ be as follows. In case (1),
  let $\k_X$ and $\k_{X_1}$ be any sections of $\kappa_X$ and
  $\kappa_{X_1}$ in $\End_A(X)$ and $\End_A(X_1)$, respectively,
  let $\irr(X,X_1)\to \rad(X,X_1)$ be any section of the
  $\k_X-\k_{X_1}$-linear canonical surjection
  $\rad(X,X_1)\to \irr(X,X_1)$ and let
  $F\colon \k \chi_0 \to \ind\,\widetilde\Gamma$ be given by the
  sections $\k_X$, $\k_{X_1}$ and $\irr(X,X_1)\to \rad(X,X_1)$. In
  this case, $(\chi_0,F)$ lies in $\Sigma_1$.  In case (2), fix
  $\k$-algebra sections $\k_X\subseteq \mathrm{End}_A(X)$ of
  $\kappa_X$, and $\k_{X_i}\subseteq \mathrm{End}_A(X_i)$ of
  $\kappa_{X_i}$, for every
  $i\in\{1,\ldots,r\}$. Consider an index $i\in \{1,\ldots,r\}$. By
  assumption on $f$, there exists a direct sum decomposition
  $\irr(X,X_i) = (\oplus_j \k_X\cdot \overline{f_{i,j}}\cdot \k_{X_i})
  \oplus M_i$
  as $\k_X\otimes_\k\k_{X_i}^\op$-modules. Let
  $\mathrm{irr}(X,X_i)\to \rad(X,X_i)$ be the following
  $\k_X\otimes_\k\k_{X_i}^\op$-linear section of the natural surjection
  $\rad(X,X_i) \twoheadrightarrow \irr(X,X_i)$ (recall that the
  algebra $\k_X\otimes_\k \k_{X_i}^\op$ is semi-simple),
  \begin{itemize}
  \item for all $j$, the composite morphism
    $\k_X\cdot\overline{f_{i,j}}\cdot \k_{X_i}\hookrightarrow
    \irr(X,X_i) \to \rad(X,X_i)$
    is the  $\k_X\otimes_\k\k_{X_i}^\op$-linear section of the composite morphism
    $\rad(X,X_i) \twoheadrightarrow \irr(X,X_i) \twoheadrightarrow
    \k_X\cdot \overline{f_{i,j}}\cdot \k_{X_i}$
    which maps $\overline{f_{i,j}}$ onto $f_{i,j}$,
    recall that the annihilator of $\overline{f_{i,j}}$ over
    $\k_X\otimes_\k\k_{X_i}^\op$ is trivial,
  \item the composite morphism $M_i \hookrightarrow \irr(X,X_i) \to
    \rad(X,X_i)$ is any $\k_X\otimes_\k\k_{X_i}^\op$-linear section of
    the composite morphism $\rad(X,X_i) \twoheadrightarrow \irr(X,X_i)
    \twoheadrightarrow M_i$.
  \end{itemize}
  These sections define a well-behaved functor
  $F\colon \k\chi_0\to \mathrm{ind}\,\Gamma$. Then
  $(\chi_0,F)\in\Sigma_1$.

  Therefore, by definition, $\Sigma_1$ is totally inductive. Hence it
  has at least one maximal element, say, $(\chi_1,F_1)$, by the
  Kuratowski-Zorn Lemma. Apply \ref{subsec_inductivewellbehaved} to
  it: since $(\chi_1,F_1)$ is maximal in $\Sigma_1$, it follows that
  $\chi_1$ is stable under predecessors in
  $\widetilde\Gamma_{\geqslant \ell_0}$.  Thus, $\Sigma_2$ contains
  $(\chi_1,F_1)$, and hence it is not empty.

  Since $\Sigma_2$ is not empty, then it is totally inductive by
  construction. Consider a maximal element $(\chi_2,F_2)$ in
  $\Sigma_2$. In this case, $\chi_2$ equals
  $\widetilde\Gamma_{\geqslant \ell_0}$. Indeed, apply
  \ref{subsec_inductivewellbehaved} to it: since $(\chi_2,F_2)$ is
  maximal in $\Sigma_2$, it follows that there is no arrow in
  $\widetilde\Gamma$ connecting a vertex in $\chi_2$ to a vertex in
  $\widetilde\Gamma_{\geqslant \ell_0} \backslash \chi_2$. Hence,
  should there exist
  $x'\in \widetilde\Gamma_{\geqslant \ell_0}\backslash\chi_2$, then
  $\chi_2\cup\{x'\}$ would be a full subquiver of $\widetilde\Gamma$
  and any $\k$-algebra section of $\kappa_{\pi x'}$ in
  $\End_A(\pi x')$ would define an extension of $F_2$ to a well-behaved
  functor $\chi_2\cup\{x'\}\to \ind A$, a contradiction to
  $(\chi_2,F_2)$ being maximal in $\Sigma_2$. Thus,
  $\chi_2=\widetilde\Gamma_{\geqslant \ell_0}$. In particular,
  $\Sigma_3$ is not empty.

  Therefore, by definition, $\Sigma_3$ is totally inductive. Consider
  a maximal element $(\chi_3,F_3)$ in $\Sigma_3$. By absurd, assume
  that the inclusion $\chi_3\subseteq \widetilde\Gamma$ is
  strict. Then, there exists $x'\in \widetilde\Gamma\backslash\chi_3$
  with $\ell(x')$ maximal because
  $\widetilde\Gamma_{\geqslant \ell_0}\subseteq \chi_3$. There is no
  arrow $x_1\to x'$ such that $x_1\in \chi_3$ because $\chi_3$ is
  stable under successors. Should there exist no arrow $x'\to x_1$ in
  $\widetilde\Gamma$ such that $x_1\in \chi_3$, then
  $\chi_3\cup\{x'\}$ would be a full subquiver of $\widetilde \Gamma$
  containing $\widetilde\Gamma_{\geqslant \ell_0}$ and stable under
  successors, and any $\k$-algebra section of $\kappa_{\pi x'}$ in
  $\End_A(\pi x')$ would define an extension of $F_3$ to a well-behaved
  functor $\chi_3\cup\{x'\}\to \ind A$, a contradiction to
  $(\chi_3,F_3)$ being maximal in $\Sigma_3$. There hence exists an
  arrow $x'\to x_1$ such that $x_1\in \chi_3$. Now, apply
  \ref{subsec_inductivewellbehaved} to $(\chi_3,F_3)$: by assumption
  on $\ell(x')$, there exists a well-behaved functor
  $F'_3\colon \chi_3'\to \ind A$ which extends $F_3$, where $\chi_3'$
  is the full subquiver of $\widetilde\Gamma$ with vertices $x'$ and
  those of $\chi_3$. The assumption on $\ell(x')$ entails that
  $(\chi_3',F_3')\in \Sigma_3$, a contradiction to $(\chi_3,F_3)$
  being maximal in $\Sigma_3$. Thus, $\chi_3=\widetilde\Gamma$.
\end{proof}

The authors acknowledge the referee for pointing out the missing
hypothesis on the annihilators in a previous version of this text. It
would be interesting to determine when such annihilators are indeed
trivial. For instance, this is the case when $\kappa_X=\k$ because,
then, $\kappa_X\otimes_\k\kappa_{X_i}^\op$ is a division algebra for
all $i$.

\subsection{Covering property of well-behaved functors}
\label{subsec_wellbehavedcovering}

Theorem~\ref{thm2}  is an adaptation of \cite[Thm. B]{MR2819689} to
 perfect fields. 
\begin{proof}[Proof of Theorem~\ref{thm2}]
  The proof uses a specific left minimal almost split morphism and a specific
  almost split sequence that arise from $F$ and which are now 
  introduced.  Let $X=\pi x$. For every arrow in $\widetilde\Gamma$
  with source $x$ (say, with target $x'$), fix one basis over
  $\kappa_{\pi x'}$ of $\mathrm{irr}(\pi x,\pi x')$. Putting 
  these bases together (for all the arrows in $\widetilde\Gamma$
  with source $x$)  yields a sequence of morphisms in
  $\k(\widetilde\Gamma)$ with domain $x$. Say, the sequence is
  $(\alpha_i)_{i=1,\ldots,r}$ where the codomain of $\alpha_i$ is
  denoted by $x_i$ (there may be repetitions in the sequence of
  codomains). Set $X_i=\pi x_i$ and $a_i=F(\alpha_i)$ for every
  $i$. In particular $a_i\colon X\to X_i$ is an irreducible morphism
  and $\overline{a_i}=\alpha_i$ if $\overline{a_i}$ is considered as
  lying in $\k(\widetilde\Gamma)(x,x_i)$. By construction,
  $[a_1,\ldots,a_r]\colon X\to \oplus_{i=1}^r X_i$ is a left minimal
  almost split morphism. If $x$ is non-injective then $a$ completes
into an almost split sequence
$0\to X\xrightarrow{a} \oplus_{i=1}^rX_i\xrightarrow{b}
  \tau_A^{-1}X\to 0$
as follows. For every $X'\in \{X_1,\ldots,X_r\}$ the family
$\{\overline{a_i}\}_{i\,s.t.\,X'=X_i}$ is a basis of $\mathrm{irr}(X,X')$
over $\kappa_{X'}$; let $\{\beta_i\}_{i\,s.t.\,X'=X_i}$ be the corresponding
dual basis of
$\mathrm{irr}(X',\tau_A^{-1}X)$ over $\kappa_{X'}$ (for the standard $\k$-modulation
of $\Gamma$);
For every $i\in\{1,\ldots,r\}$ such that $X'=X_i$  set $b_i\colon
X_i\to \tau_A^{-1}X$ to be the image of $\beta_i\in
\k(\widetilde\Gamma)(x_i,\tau^{-1}x)$ under $F$
(hence, if one considers $\overline{b_i}$ as lying in
$\k(\widetilde\Gamma)(x_i,\tau^{-1}x)$ then  $\overline{b_i}=\beta_i$). By construction
$\gamma_{\tau^{-1}x}= \sum_{i=1}^r\overline{a_i}\otimes
\overline{b_i}$. Therefore $\sum_{i=1}^ra_ib_i=0$ because $F$ is
well-behaved and maps each $\overline{a_i}$ and $\overline{b_i}$ to
$a_i$ and $b_i$,
respectively. Since moreover $a$ is left minimal  almost split, the
morphism $b=[b_1,\ldots,b_r]^t$ is
right minimal almost split. Thus $(a,b)$ forms the announced almost
split sequence.

\medskip

(a) The two maps are dual to each other so only the first one is taken
care of.

The surjectivity for every $x$ is proved by induction on $n\geqslant
0$. If $n=0$ it follows from: $\rad^0(Fx,Fy)/\rad(Fx,Fy)$ is $\kappa_{Fx}$ or $0$ according to
whether $Fx=Fy$ or $Fx\neq Fy$, and $\mathcal
R^0\k(\widetilde\Gamma)(x,z)/\mathcal R\k(\widetilde\Gamma)(x,z)$ is $\kappa_x$ or $0$ according to
whether $x=z$ or $x\neq z$. If $n\geqslant 1$, if the surjectivity
is already proved for indices smaller that $n$, and if $f\in \rad^n(Fx,Fy)$ is given, then there exists $(u_i)_i\in \oplus_{i=1}^r
\rad^{n-1}(Fx_i,Fy)$ such that $f=\sum_ia_iu_i$ (\ref{subsec_factorisation}); for every $i$, 
there exists $(\beta_{i,z})_z\in \oplus_{Fz=Fy}\mathcal
R^{n-1}\k(\widetilde\Gamma)(x_i,z)$ such that
$u_i=\sum_zF(\beta_{i,z})\,\mathrm{mod}\,\rad^n$ (induction hypothesis); thus
$f-\sum_zF\left( \sum_i \alpha_i\beta_{i,z}\right)\in \rad^{n+1}$. This proves the surjectivity at index $n$.

The injectivity for every $x$ is also proved by induction on
$n\geqslant 0$. If $n=0$ it follows from:
$\k(\widetilde\Gamma)(x,z)=\mathcal R\k(\widetilde\Gamma)(x,z)$ if
$x\neq z$, and $\k(\widetilde\Gamma)(x,x)=\kappa_x=\kappa_{Fx}$, and
$F$ induces a section $\kappa_x\to \mathrm{End}_A(Fx)$ of the
canonical surjection $\mathrm{End}_A(Fx)\to \kappa_{Fx}$. Let
$n\geqslant 1$. Assume the injectivity for indices smaller that
$n$. Let
$(\phi_z)\in \oplus_{Fz=Fy}\mathfrak R^n\k(\widetilde\Gamma)(x,z)$ be
such that $\sum_z F(\phi_z)\in \rad^{n+1}(Fx,Fy)$. Using the
surjectivity and \ref{subsec_factorisation} yields
$(\psi_z)_z\in \oplus_{Fz=Fy}\mathfrak
R^{n+1}\k(\widetilde\Gamma)(x,z)$
together with $(u_i)_i\in \oplus_{i=1}^r\rad^{n+1}(Fx_i,Fy)$ such that
$\sum_zF(\phi_z-\psi_z)=\sum_i a_iu_i$. On the other hand,
$n\geqslant 1$ and $\{\alpha_j\}_{j\in \{1,\ldots,r\}}$ contains a
basis of $M(x,x_i)$, which is part of the induced modulation of
$\widetilde\Gamma$, over $\kappa_{x_i}$ for every
$i\in\{1,\ldots,r\}$.  The construction of $\k(\widetilde\Gamma)$
therefore yields
$(\theta_{i,z})_i \in \oplus_{i=1}^r\k(\widetilde\Gamma)(x_i,z)$ such
that $\phi_z-\psi_z=\sum_i\alpha_i\theta_{i,z}$, for every
$z$. Putting these morphisms together and using that $a_i=F(\alpha_i)$
for every $i$ yields
$\sum_ia_i\left(\sum_zF(\theta_{z,i})-u_i\right)=0$. Now distinguish
two cases according to whether $x$ is injective or not. If $x$ is
injective then $\sum_zF(\theta_{i,z})=u_i$ which, following the
induction hypothesis, implies that
$\theta_{i,z}\in \mathfrak R^n\k(\widetilde\Gamma)$ for every $i$ and
every $z$. Thus
$\phi_z=\psi_z+\sum_i\alpha_i\theta_{i,z}\in \mathfrak
R^{n+1}\k(\widetilde\Gamma)$
for every $z$. If $x$ is not injective there exists
$v\in \mathrm{Hom}_A(\tau_A^{-1} Fx,Fy)$ such that
$\sum_zF(\theta_{i,z})-u_i=b_iv$ for every $i$.  Using the
surjectivity yields
$(\chi_z)_z\in \oplus_{Fz=Fy}\k(\widetilde\Gamma)(\tau^{-1}x, z)$ such
that $v=\sum_zF(\chi_z)\,\mathrm{mod}\,\rad^{n-1}$. In particular
$b_iv=\sum_zF(\beta_i\chi_z)\,\mathrm{mod}\,\rad^n$ for every
$i$. Hence
$\sum_zF(\theta_{i,z}-\beta_i\chi_z) = u_i\,\mathrm{mod}\,\rad^n$.
Therefore
$\theta_{i,z}-\beta_i\chi_z\in \mathfrak R^n\k(\widetilde\Gamma)$ for
every $i$ and every $z$ (by induction and because $u_i\in\rad^{n+1}$).
Since moreover $\sum_i\alpha_i\beta_i=0$,
$\psi_z\in\mathfrak R^{n+1}\k(\widetilde\Gamma)$ and
$\phi_z=\psi_z+\sum_i\alpha_i\theta_{i,z}$ for every $z$, it follows
that $\phi_z\in\mathfrak R^{n+1}\k(\widetilde\Gamma)$.

\medskip

(b) follows from (a) and from \ref{subsec_radical} (part (b)).

\medskip

(c) follows from (a), (b) and the fact that $\Gamma$ is generalised
standard, that is, $\bigcap_{n\geqslant 0}\rad^n(X,Y)=0$ for every
$X,Y\in \Gamma$.
\end{proof}

\section{Application to compositions of irreducible morphisms}
\label{sec:applications1}

Recall the convention set in the introduction: whenever
$f \colon X \to Y$ and $g \colon Y \to Z$ are two mappings, their
composition is denoted by $fg \colon X \to Z$. The following
equivalence was proved in \cite[Prop. 5.1]{MR2819689} when $\k$ is
algebraically closed and under the additional assumption that the
valuation of the involved arrows are trivial. This last assumption is
dropped here.
\begin{prop}
  Let $X_1,\ldots,X_{n+1}\in \mathrm{ind}\,A$. The following conditions
  are equivalent
  \begin{enumerate}[(a)]
  \item there exist irreducible morphisms
    $X_1\xrightarrow{h_1}\cdots \xrightarrow{h_n}X_{n+1}$ such that $h_1\cdots h_n\in \rad^{n+1}\backslash\{0\}$,
% MODIFICATION: the statement with and's, or's, either's seemed
% unclear to me. I tried to modify it to make it clearer.
  \item there exist irreducible morphisms $f_i\colon X_i\to X_{i+1}$
    and morphisms $\varepsilon_i\colon X_i\to X_{i+1}$, for
    every $i$, such
    that $f_1\cdots f_n=0$, such that
    $\varepsilon_1\cdots\varepsilon_n\neq 0$ and such that, for every $i$,  either $\varepsilon_i\in \rad^2$ or else $\varepsilon_i=f_i$.
  \end{enumerate}
\end{prop}
\begin{proof}
  The implication $(b)\Rightarrow (a)$ was proved in
  \cite[Thm. 2.7]{MR2578589} (the proof there works for artin algebras and the standard hypothesis
  made there plays no role for this implication).
Assume $(a)$. Let $\Gamma$ be the component of $\Gamma(\mathrm{mod}\,A)$
containing $X_1,\ldots,X_{n+1}$, let $\pi\colon \tilde\Gamma\to\Gamma$ be
the universal covering and $F\colon \k(\widetilde\Gamma)\to \mathrm{ind}\,\Gamma$ be a well-behaved functor
(\ref{subsec_wellbehavedexistence}). Let $x_1\in \pi^{-1}(X_1)$. 
There is a unique path $\gamma\colon x_1\to x_2\to \cdots\to x_{n+1}$ in
$\widetilde\Gamma$ which image under $\pi$ is $X_1\to X_2\to
\cdots\to X_{n+1}$. Let $h_i\colon X_i\to X_{i+1}$ ($1\leqslant
i\leqslant n$) be irreducible morphisms such that $h_1\cdots h_n\in
\rad^{n+1}\backslash\{0\}$ and consider $\overline{h_i}\in \mathrm{irr}(X_i,X_{i+1})$ as lying in
$\k(\widetilde\Gamma)(x_i,x_{i+1})$. Let $h_i'=h_i-F(\overline{h_i})$
for $1\leqslant i\leqslant n$. Then $h_i'\in \rad^2$ because $F$
is well-behaved. Therefore
$F(\overline{h_1}\cdots\overline{h_n})\in\rad^{n+1}$. Since
$\mathfrak R^{n+1}\k(\widetilde\Gamma)(x_1,x_{n+1})=0$  (the path
$\gamma$ has length $n$, \ref{subsec_radical}), it follows that
$\overline{h_1}\cdots\overline{h_n}=0$
(\ref{subsec_wellbehavedcovering}). This and  $h_1\cdots
h_n\neq 0$ imply that $F(\overline{h_1}\cdots \overline{h_n})-h_1\cdots
h_n\neq 0$ that is,
the sum of the morphisms
\[
F(\overline{h_1})\cdots
  F(\overline{h_{i_1-1}}) h_{i_1}'F(\overline{h_{i_1+1}})\cdots
  F(\overline{h_{i_t-1}})h_{i_t}'F(\overline{h_{i_t+1}})\cdots
  F(\overline{h_n})\,,
  \]
for $t\in\{1,\ldots,n\}$ and $1\leqslant i_1<\cdots <i_t\leqslant n$,
is non-zero.
  Hence there exists $t\in \{1,\ldots,n\}$ and $1\leqslant
i_1<\cdots<i_t\leqslant n$ such that the corresponding term in the
above sum is non-zero. Define $f_j:=F(\overline{h_j})$, and
$\varepsilon_j:=F(\overline{h_j})$ if $j\not\in \{i_1,\ldots,i_t\}$ or
$\varepsilon_j:=h_j'$ if $j\in \{i_1,\ldots,i_t\}$. Then
$\{f_i,\varepsilon_i\}_{i=1,\ldots,n}$ fits the
  requirements of $(b)$.
\end{proof}

\section{Acknowledgements}

Part of the work presented in this text was done while the second
named author was visiting Universit\'e de Sherbrooke (Qu\'ebec,
Canada). He thanks Ibrahim Assem and the Department of Mathematics
there for the warm hospitality. The authors thank the referee for
several suggestions improving the quality of the presentation and the
proofs.

\bibliographystyle{plain}
\bibliography{biblio-CLT12}
\end{document}